\documentclass{aptpub}
\usepackage{amsfonts, amsmath, amssymb, url,mathtools}
\usepackage{graphicx}
\usepackage[english] {babel}
\usepackage{cellspace}
\numberwithin{equation}{section}

\usepackage[ps, all,cmtip, dvips]{xy}

\newcommand{\R}{\mathbb{R}}

\newcommand{\Z}{\mathbb{Z}}
\newcommand{\La}{\mathbb{L}}

\newcommand{\Q}{\mathbb{Q}}
\newcommand{\eps}{\varepsilon}
\DeclareMathOperator{\conv}{conv}

\DeclareMathOperator{\su}{sup}

\DeclareMathOperator{\indre}{int}

\DeclareMathOperator{\diam}{diam}
\newcommand{\Ha}{\mathcal{H}}

\DeclareRobustCommand\altoverline[2][0.8]{%
\mathmakebox[\widthof{#2}][c]{\overline{\mathmakebox[\widthof{#2}*\real{#1}][c]{#2}}}}

\authornames{Anne Marie Svane} % insert the authors here for use in running head
\shorttitle{Local digital estimators of intrinsic volumes} % insert short title here for use in running head

\begin{document}
\title{Local digital estimators of intrinsic volumes for Boolean models and in the design based setting} % insert title - use \\ if it requires more than one line.

\authorone[Aarhus University]{Anne Marie Svane} % Affiliation is just the name of your university or institution

\addressone{Department of Mathematics, Aarhus University, 8000 Aarhus C, Denmark} % Your postal address goes here.

\begin{abstract}
In order to estimate the specific intrinsic volumes of a planar Boolean model from a binary image, we consider local digital algorithms based on weighted sums of $2\times 2$ configuration counts.  For Boolean models with balls as grains, explicit formulas for the bias of such algorithms are derived, resulting in a set of linear equations that the weights must satisfy in order to minimize the bias in high resolution. These results generalize to larger classes of random sets, as well as to the design based situation, where a fixed set is observed on a stationary isotropic lattice. Finally, the formulas for the bias obtained for Boolean models are applied to existing algorithms in order to compare their accuracy.
\end{abstract}

 % insert keywords separated by a semicolon

 % insert the primary Maths Subject Classification number in the first bracket
         % and the secondary ams number(s) in the second bracket
         % e.g. \ams{60E20}{49G03;49F10}

\section{Introduction}\label{intro}
Let $X \subseteq \R^2$ be a compact subset of the plane. Suppose we are given a digital image of $X$, i.e.\ the only information about $X$ available to us is the set  $X\cap \La$ where $\La \subseteq \R^2$ is a square lattice. In the language of signal processing, we are thus using an \emph{ideal sampler} to obtain a sample of the characteristic function of $X$ at all the points of $\La$. In image analysis terms, $\La$ can be interpreted as the set of all pixel midpoints and the digitization $X\cap \La $ contains the same information about $X$ as the commonly used Gauss digitization \cite[p. 56]{digital}. From this binary representation of $X$, we would like to recover certain geometric properties of $X$. The quantities we are interested in are the so-called intrinsic volumes $V_i$. In the plane, these are simply the volume $V_2(X)$, the boundary length $2V_1(X)$, and the Euler characteristic $V_0(X)$. See \cite[Chapter 4]{schneider} for the definition when $X$ is polyconvex.

In this paper, we exclusively consider local digital estimators based on $2\times 2$ confi\-gu\-ration counts in a square lattice. Motivated by the additivity of intrinsic volumes, these are defined as follows: The plane is divided into a disjoint union of square cells with vertices in $\La$. For each $2\times 2$ cell in the lattice, each vertex may belong to either $X$ or $\R^2\backslash X$, yielding $2^4=16$ different possible configurations. Each cell contributes to the estimator for $V_i(X)$ with a certain weight depending only on the configuration. Thus the estimator becomes a weighted sum of the configuration counts. The weights can in principle be chosen freely. Algorithms of this type are desirable as they are simple and efficiently implementable based on linearly filtering the image.  

One way of testing the quality of local algorithms is by simulations on a fixed test set for various high resolutions, see e.g.\ \cite[Section 10.3.4]{digital}. 
In contrast, we shall follow Ohser, Nagel, and Schladitz in \cite{nagel}, where the algorithms are applied to a standard model from stochastic geometry, namely the Boolean model. But rather than testing a known algorithm, we let the weights be arbitrary and derive conditions on the weights such that the bias of the estimator is minimal for high resolutions. 

If the grains are almost surely balls, a Steiner-type result for finite sets shown by Kampf and Kiderlen in \cite{markus} yields a general formula for the estimator from which the asymptotic behaviour can be derived. 
The main result is that a local estimator is asymptotically unbiased if and only if the weights satisfy certain linear equations. Moreover, we obtain formulas for the approximate bias in high resolution. These results are stated in Theorem~\ref{w1} and~\ref{w2} below. 

Local estimators are introduced in Section \ref{notation}.
This is specialized to Boolean mo\-dels in Section \ref{boolean} and the computations are performed in Section \ref{unbiased}. 

In Section \ref{general}, the main theorems are generalized to a larger class of Boolean models where the grains allow a ball of radius $\eps>0$ to slide freely. A formula by Kiderlen and Jensen presented in \cite{eva} also yields an immediate generalization of the first-order results to general standard random sets, see Section \ref{standard}.

We then turn to the design based situation where a deterministic set $X$ is observed on a randomly translated and rotated lattice. Under certain conditions on $X$, we obtain a generalization of the main theorems for Boolean models. This is done for the boundary length in Section \ref{bound}, using a result of Kiderlen and Rataj from \cite{rataj}, and for the Euler characteristic in Section \ref{euler} by a refinement of their approach.  

In the literature, various algorithms for computing intrinsic volumes are suggested.  The obtained formulas allow for a computation of the bias in high resolution and hence a comparison of the commonly used algorithms. This is the content of the last section of the paper, Section \ref{classical}.

\section{Local digital estimators}\label{notation}
%In this paper we study local estimators of the intrinsic volumes of subsets of the plane from a digital image. That is, suppose $L$ is a lattice in the plane(). Then the digital image is the set $X\cap L$(foreground, black) and the background(white) $X\cap L$. 
Let $\Z^2$ be the standard lattice in $\R^2$. 
Let $C$ denote the unit square $[0,1]\times [0,1]$ in $\R^2$ and let $C_0$ be the set of vertices in $C$.  We enumerate the elements of $C_0$ as follows: $x_0=(0,0)$, $x_1=(1,0)$, $x_2=(0,1)$, and $x_3=(1,1)$. A configuration is a subset $\xi \subseteq C_0$. We denote the $16$ possible configurations by $\xi_l$, $l=0,\dots,15$, where the configuration $\xi$ is assigned the index 
\begin{equation*}
l=\sum_{i=0}^3 2^i 1_{x_i\in \xi}.
\end{equation*}
Here $1_{x_i \in \xi}$ is the indicator function.

More generally, we shall consider an orthogonal lattice $a\La=aR_v(\Z^2 + c)$ where $c\in C$ is a translation vector, $R_v$ is the rotation by the angle $v\in [0,2\pi ]$, and $a>0$ is the lattice distance. The configuration $\xi_l$ is then understood to be the corresponding transformation $aR_v(\xi_l + c)$ of the configuration $\xi_l\subseteq \Z^2$. 

The elements of $\xi_l$ are referred to as the `foreground' or `black' pixels and will also sometimes be denoted by $B_l$, while the points in the complement $W_l=C_0\backslash \xi_l =\xi_{15-l}$ are referred to as the `background' or `white' pixels.

%We shall mainly be working in a rotation and translation invariant setting. Hence, we will not be able to  distinguish between configurations that differ only by a rigid motion. 
The $16$ possible configurations are divided into six equivalence classes under rigid motions. These are denoted by $\eta_j$ for $j = 1,\dots , 6$. These are defined in Table \ref{config}.
\begin{table}
\begin{center}
\begin{tabular}{ScScScScSc} 
\hline
$j$ &$\eta_j$ & $d_j$ & Description& Example\\
\hline
1 & $\{\xi_0\}$ & 1 & 4 white vertices& $\begin{bmatrix}\circ & \circ \\ \circ & \circ \end{bmatrix}$\\
2 & $\{\xi_1,\xi_2,\xi_4,\xi_8\} $ & 4 & 3 white and 1 black vertices &$\begin{bmatrix}\circ & \circ \\ \bullet & \circ \end{bmatrix}$\\
3 & $\{\xi_3,\xi_5,\xi_{10},\xi_{12}\} $ & 4 & 2 adjacent white and 2 black vertices &$\begin{bmatrix}\circ & \circ \\ \bullet & \bullet \end{bmatrix}$\\
4 & $\{\xi_6,\xi_9\}$ & 2 & 2 opposite white and 2 black vertices &$\begin{bmatrix}\circ & \bullet \\ \bullet & \circ \end{bmatrix}$\\
5 & $\{\xi_7,\xi_{11},\xi_{13},\xi_{14}\} $ & 4 & 1 white 3  black vertices &$\begin{bmatrix}\bullet & \circ \\ \bullet & \bullet \end{bmatrix}$\\
6 & $\{\xi_{15}\}$ & 1 & 4 black vertices &$\begin{bmatrix}\bullet & \bullet \\ \bullet & \bullet \end{bmatrix}$\\
\hline
\end{tabular}
\end{center}
\caption{Configuration classes}\label{config}
\end{table}
The number $d_j$ is the number of elements in the equivalence class $\eta_j$.

Now let $X\subseteq \R^2$ be a compact set. Suppose we observe $X$ on the  lattice $a\La $. Based on the set $X\cap a\La$ we want to estimate the intrinsic volumes $V_i$ introduced in Section~\ref{intro}.

In order for the $V_i$ to be well-defined and for the digitization $X\cap a\La$ to carry enough information about $X$, we require that $X$ is sufficiently `nice'. The notion of a gentle set is introduced in Section \ref{bound} when dealing with $V_1$. This includes all topologically regular polyconvex sets. When we work with $V_0$, $X$ will be assumed to be either a compact topologically regular polyconvex set or a compact full-dimensional $C^2$ manifold. A set is called topologically regular if it coincides with the closure of its interior.

Our approach is to consider a local algorithm based on the observations of $X$ on the $2\times 2$ cells of $a\La$.
By additivity of the intrinsic volumes, $V_i(X)$ is a sum of contributions from each lattice cell $z+aR_v(C)$ for $z\in a\La$. We estimate this by a certain weight $w^{(i)}(a,z)$, depending only on the information we have about the cell, i.e.\ the configuration 
\begin{equation*}
X\cap (z+aR_v(C_0))-(z-c)=(X -(z-c)) \cap \xi_{15}.
\end{equation*}
Recall here that $\xi_{15}=aR_v(C_0+c)$ is the set of vertices in the unit cell of $a\La$.

Since $V_i$ is invariant under rigid motions, we would like the estimator to satisfy
\begin{equation*}
\hat{V}_i(X)=\hat{V}_i(MX)
\end{equation*}
for any rigid motion $M$ preserving $a\La$. Thus  $w^{(i)}(a,z)$ should only depend on the equivalence class $\eta_j$ of $(X-(z-c)) \cap \xi_{15}$ under rigid motions.

As $V_i$ is homogeneous of degree $i$, i.e.\ $V_i(aX)=a^iV_i(X)$, the estimator should also satisfy
\begin{equation*}
\hat{V}_i(aX\cap a\La)=a^i\hat{V}_i(X\cap \La).
\end{equation*}
We therefore assume $w^{(i)}(a,z)=a^iw_j^{(i)}$ where $w_j^{(i)}\in \R$ are constants.

We are thus led to consider estimators of the form
\begin{equation*}
\hat{V}_i(X) =  a^i\sum_{j=1}^6  w_j^{(i)} N_j
\end{equation*}
where $N_j$ is the number of occurrences of the configuration class $\eta_j$ 
\begin{equation*}%\label{indic}
N_j = \sum_{z \in a\La } 1_{(X-(z-c)) \cap \xi_{15} \in \eta_j}.
\end{equation*}
%Note how $N_j$ only depends on $X\cap a\La$, as 
%\begin{equation*}
%(X-z) \cap \xi_{15}=((X\cap a\La) -z) \cap \xi_{15}.
%\end{equation*}

It is also natural to require the estimators to be compatible with interchanging background and foreground as follows:
\begin{align}
\hat{V}_1(X)&=\hat{V}_1({\R^2\backslash X}),\label{item1}\\
\hat{V}_0(X)&=-\hat{V}_0({\R^2\backslash X}).\label{item2}
\end{align}
The reason for the first condition is that interchanging foreground and background does not change the boundary. The second condition is natural because the Euler characteristic satisfies
\begin{equation*}
{V}_0(X)=-{V}_0(\overline{\R^2\backslash X})
\end{equation*}
for both topologically regular compact polyconvex sets, see \cite{ons}, and compact 2-ma\-ni\-folds with boundary.

\section{The 2D Boolean model}\label{boolean}
Throughout this paper, a Boolean model $\Xi$ will mean  a stationary isotropic Boolean model in the plane with compact convex grains and intensity $\gamma$. That is,
\begin{equation*}
\Xi = \bigcup_{i}(x_i + K_i)
\end{equation*}
where $\{x_1, x_2, \dots \}$ is a stationary Poisson process in $\R^2$ with intensity $\gamma $ and $K_1,K_2,\dots $ is a sequence of i.i.d.\ random compact convex sets in $\R^2$ with rotation invariant distribution $\Q$ satisfying $EV_i(K) < \infty$ for $i=0,1,2$. %We assume throughout that the grain distribution $\Q$ is rotation invariant and hence $\Xi$ is isotropic. 
See e.g.\ \cite{SW} for more details. 
%For technical reasons, we will assume throughout that there is an $\varepsilon >0$ such that $r> \varepsilon $ a.\ s. 

The specific intrinsic volumes of a Boolean model are defined by 
\begin{equation}\label{defVi}
\altoverline{V}_i(\Xi )= \lim_{r\to \infty} \frac{EV_i(\Xi \cap rW)}{V_2(rW)}
\end{equation}
where $W$ is any compact convex set with non-empty interior, see \cite[Theorem 9.2.1]{SW}.

Now assume that we observe $\Xi $ on a lattice $a\La$ in a compact convex window $W$ with non-empty interior. By the isotropy assumption, we may as well assume the lattice to be the standard lattice $a\Z^2$. Thus we observe the set $\Xi \cap a\Z^2 \cap W$. %Based on this, we want to define local estimators for the specific intrinsic volumes. 

%The limit in \eqref{defVi} is introduced to correct for edge effects. However, we only observe in a bounded window. We can get rid of the limit as follows:
Let $C_z = z + aC$ be a lattice cell with $z\in a\Z^2$.  Write 
\begin{equation*}
V_{i,z} =V_i(C_z\cap \Xi)-V_{i}(\partial^+C_z\cap \Xi)
\end{equation*}
where $\partial^+C_z = z + a([0,1]\times \{1\}\cup\{1\}\times [0,1])$ is the upper right boundary.
Then  \cite[Theorem 9.2.1.]{SW} implies that $ EV_{i,z} =a^2\altoverline{V}_i(\Xi)$.
A summation over all lattice cells contained in $W$ yields
\begin{equation}\label{partial+}
\altoverline{V}_i(\Xi)=\sum_{z\in a\Z^2\cap (W \ominus a\check{C})}\frac{EV_{i,z}}{V_2(C_z)N_0}=\sum_{z\in a\Z^2\cap (W \ominus a\check{C})}\frac{EV_{i,z}}{a^2N_0}.
\end{equation}
 where $\check{C}=\{-x\mid x\in C\}$ and $W\ominus a\check{C} = \{x \in \R^2 \mid x+ aC \subseteq W \}$ and $N_0$ is the total number of points in $a\Z^2\cap (W \ominus a\check{C})$. 
%Thus $a\Z^2\cap (W \ominus a\check{C})$ contains exactly those $z$ such that $C_z$ is contained in $W$. 

%Define 
%\begin{equation*}
%W_a = \bigcup_{z\in a\Z^2\cap (W \ominus a\check{C})}C_z \text{ and } \partial^+ W_a = \partial W_a \cap \bigcup_{z\in a\Z^2\cap (W \ominus a\check{C})} \partial^+ C_z,
%\end{equation*}
% i.e.\ $W_a$ is the union of the lattice squares lying entirely in $W$. Then \eqref{partial+} has the following interpretation:
% \begin{equation*}
% \overline{V}_i(\Xi)=\frac{EV_i(W_a\cap \Xi) - EV_i(\partial^+ W_a \cap \Xi)}{V_2(W_a)}. 
%\end{equation*}

As in Section \ref{notation}, we estimate each contribution $EV_{i,z}$ by a weight of the form  $a^i w_j^{(i)}$ depending on the configuration type $\eta_j$. Then \eqref{partial+} yields an estimator of the form
\begin{equation}\label{Nest}
\hat{V}_i(\Xi) = a^{i-2}\sum_{j=1}^6  w_j^{(i)}  \frac{N_j}{N_0}
\end{equation} 
where $w_j^{(i)}\in \R$ are arbitrary weights and the number of configurations $N_j$ are given by
\begin{equation}\label{indik}
N_j = \sum_{z \in a\Z^2\cap (W \ominus a\check{C})} 1_{(\Xi-z) \cap \xi_{15} \in \eta_j}.
\end{equation}
%
%The natural approximation of \eqref{defVi} based only on the knowledge of the local configurations is 
%\begin{equation*}
%\hat{V}_i(\Xi) = \sum_{j=1}^6  a^{i}w_j \frac{N_j}{V_2(W)}
%\end{equation*} 
%where $w_j\in \R$ are arbitrary weights and the number of configurations $N_j$ is
%\begin{equation}\label{indik}
%N_j = \sum_{z \in a\Z^2\cap (W \ominus a\check{C})} 1_{(\Xi-z) \cap \xi_{15} \in \eta_j}.
%\end{equation}
%Here $\check{C}=\{-x\mid x\in C\}$ and $W\ominus a\check{C} = \{x \in \R^2 \mid x+ aC \subseteq W \}$.

%As opposed to the approach in \cite{nagel}, we make no a priori assumptions on the weights but leave them arbitrary and investigate the behavior of the estimator.
Ideally, $\hat{V}_i$ would define an unbiased estimator, i.e.\ $E\hat{V}_i(\Xi) = \altoverline{V}_i(\Xi)$. Generally, this is not possible with finite resolution, i.e.\ when $a>0$. Instead, we shall obtain conditions for this to hold asymptotically when the lattice distance tends to zero:
\begin{equation*}
\lim_{a\to 0} E\hat{V}_i(\Xi) = \altoverline{V}_i(\Xi).
\end{equation*}

The mean value of $\hat{V}_i(\Xi)$ is 
\begin{equation}\label{Vhat}
E\hat{V}_i(\Xi) = a^{i-2}\sum_{j=1}^6 w_j^{(i)}  E\bigg(\frac{N_j}{N_0}\bigg) = a^{i-2}\sum_{j=1}^6  w_j^{(i)}  P(\Xi \cap aC_0 \in \eta_j)
\end{equation}
by \eqref{indik} and stationarity of $\Xi$.

%Let $\xi_{l_j} \in \eta_j$ be a representative. Then there are formulas
For each $\xi_l $, there are formulas of the form
\begin{equation}\label{PXi}
P(\Xi \cap aC_0 = \xi_{l}) = \sum_{k=0}^{15} b_{lk} P(\xi_k \subseteq \R^2\backslash \Xi)
\end{equation}
for suitable integers $b_{lk}$, see also \cite{nagel}.
%where $\zeta_l \subseteq C_0$ is the subset with
%\begin{equation*}
%l = \sum_{i=0}^3 2^i 1_{x_i \in \zeta_l}
%\end{equation*}
As $\Xi$ is stationary and isotropic, $P(\Xi \cap aC_0 = \xi_{l})$  and $P(\xi_k \subseteq \R^2\backslash \Xi)$  depend only on $\xi_l$ and $\xi_k$ up to rigid motions. Let $\xi_{k_i} $ and $\xi_{l_j}$ be  representatives for $\eta_i$ and $\eta_j$, respectively. Then \eqref{PXi} reduces to
\begin{equation}\label{bprime}
P(\Xi \cap aC_0 = \xi_{l_j}) = \sum_{i=1}^{6} b_{ij}' P(\xi_{k_i} \subseteq \R^2\backslash \Xi)
\end{equation}
with the integer $b_{ij}'$ given as the ${ij}$th entry in the matrix
\begin{equation*}
B=\begin{pmatrix}
0&0&0&0&0&1\\
0&0&0&0&1&-4\\
0&0&1&0&-2&4\\
0&0&0&1&-1&2\\
0&1&-2&-2&3&-4\\
1&-1&1&1&-1&1
\end{pmatrix}.
\end{equation*}

The right hand side of \eqref{bprime} is now well-known, since
\begin{equation}\label{Pzeta}
P(\xi_k \subseteq \R^2 \backslash \Xi) = e^{-\gamma EV_2(\xi_k \oplus K)}
\end{equation}
 where $K$ is a random compact convex set of distribution $\Q$ and $\oplus $ denotes Minkowski addition, see \cite{SW}.  Thus we must compute $EV_2(\xi_k \oplus K)$.
 
If $F_k=\conv(\xi_k )$ denotes the convex hull of $\xi_k$, an application of the rotational mean value formula, see \cite[Theorem 6.1.1]{SW}, shows that
\begin{equation}\label{steiner}
EV_2(F_k \oplus K) = EV_2(K) + \tfrac{2}{\pi}V_1(F_k)EV_1(K) + V_2(F_k),
\end{equation}
since the grain distribution is isotropic. It remains to compute the error 
\begin{equation}\label{error}
EV_2(F_k \oplus K)- EV_2(\xi_k \oplus K).
\end{equation}

\section{Boolean models with random balls as grains}\label{unbiased}
We first restrict ourselves to Boolean models where the grains are a.\ s.\ balls $B(r)$ of random radius $r$. For technical reasons we will assume throughout this section that there is an $\eps>0$ such that $r\geq \eps$ a.\ s. 

In  \cite[Proposition 1]{markus}, Kampf and Kiderlen give an expression for the error \eqref{error}. Applied to our situation, this becomes a power series in $\frac{a}{r}$:
\begin{align}%\nonumber
V_2(F_k \oplus B(r))- V_2(\xi_k \oplus B(r)) 
%=& a^2V_2\left( (a^{-1}F_k) \oplus B\left(\tfrac{r}{a}\right)\right)- a^2V_2\left((a^{-1}\xi_k) \oplus B\left(\tfrac{r}{a}\right)\right)\\ 
= 2a^2\sum_{n=1}^{\infty}\frac{(2n-3)!!}{(2n)!!}V_1^{(2n+1)}(a^{-1}\xi_k)\left(\frac{a}{r}\right)^{2n-1}\label{power}
\end{align}
whenever $\frac{a}{r}$ is sufficiently small. Since $a^{-1}\xi_k$ is independent of $a$, the $V_1^{(2n+1)}(a^{-1}\xi_k)$ are constants. These are called intrinsic power volumes in \cite{markus} and are given by
\begin{equation*}
V_1^{(m)}(\xi_k)= \frac{1}{m2^{m-1}}\sum_{F\in \mathcal{F}_1(F_k)} \gamma(F_k,F) V_1(F)^m
\end{equation*}
where $\mathcal{F}_1(F_k)$ is the set of 1-dimensional faces of $F_k$ and $\gamma(F_k,F)$ is the outer angle which in $\R^2$ is just $(\text{dim}F_k)^{-1}$. See \cite{markus} for the definition of the double factorial.

The condition $r\geq \varepsilon$ a.\ s.\ ensures that whenever $a$ is sufficiently small, \eqref{power} holds a.\ s.
Combining this with \eqref{steiner}, we obtain a power series expansion
\begin{align*}
EV_2(\xi_k\oplus B(r)) = {}&EV_2(B(r)) + a\tfrac{2}{\pi}V_1(a^{-1}F_k)EV_1(B(r)) + a^2V_2(a^{-1}F_k) \\&- a^3V_1^{(3)}(a^{-1}\xi_k) E({r}^{-1}) + O(a^5).
\end{align*}
Computing the constants $V_i(a^{-1}F_k)$ and $V_1^{(3)}(a^{-1}\xi_k)$ directly and  
in\-ser\-ting in the Taylor expansion for the exponential function in \eqref{Pzeta}, shows that $P(\xi_k  \subseteq  \R^2 \backslash \Xi)$ is given by a power series 
\begin{align}\label{pexp}
 c_1 + {}&\big( c_2 + ac_3 \tfrac{\gamma}{\pi} EV_1(B(r))  + a^2\big(c_4\gamma +c_5\big(\tfrac{\gamma}{\pi}EV_1(B(r))\big)^2\big) \\
 &+a^3\big(c_6\gamma E(r^{-1}) + c_7\tfrac{\gamma^2}{\pi}EV_1(B(r))+ c_8\big(\tfrac{\gamma}{\pi} EV_1(B(r))\big)^3\big) \big) e^{-\gamma EV_2(B(r))}+ O(a^4) \nonumber
\end{align}
for $a$ sufficiently small and constants $c_1,\dots , c_8$ depending on $k$. If $\xi_{k_j} $ is a representative for $\eta_j$, define $A$ to be the matrix with entry $a_{mj}$
the constant $c_{m}$ occurring in the formula for $P(\xi_{k_{j}} \subseteq  \R^2 \backslash \Xi)$ for $j=1,\dots ,6$. A direct computation shows that
\begin{equation*}
A=\begin{pmatrix}
1&0&0&0&0&0\\
0&1&1&1&1&1\\
0&0&-2&-2\sqrt{2}&-(2+\sqrt{2})&-4\\
0&0&0&0&-\frac{1}{2}&-1\\
0&0&2&4&3+2\sqrt{2}&8\\
0&0&\frac{1}{12}&\frac{\sqrt{2}}{6}&\frac{\sqrt{2}+1}{12}&\frac{1}{6}\\
0&0&0&0&\frac{2+\sqrt{2}}{2}&4\\
0&0&-\frac{4}{3}&-\frac{8\sqrt{2}}{3}&-\frac{10+7\sqrt{2}}{3}&-\frac{32}{3}
\end{pmatrix}.
\end{equation*}

Inserting this in \eqref{bprime}, we obtain expressions for $P(\Xi \cap aC_0 = \xi_{l_j})$  of the form \eqref{pexp} with  constants $c_m$ given by the $j$th column in $AB$. 
%\begin{equation*}
%Here $w^{(i)} = (w_1^{(i)},\dots, w_6^{(i)})^T$ is the vector of weights. 
Then  by \eqref{Vhat}, $a^{2-i}E\hat{V}_i(\Xi)$ is also of the form \eqref{pexp} with vector of constants $c^{(i)}=(c^{(i)}_1,\dots ,c^{(i)}_8)$ given by 
\begin{equation*}
(c^{(i)})^T=ABD(w^{(i)})^T
\end{equation*}
where $w^{(i)}=(w_1^{(i)},\dots , w_6^{(i)})$ is the vector of weights and $D$ is the diagonal matrix with $j$th diagonal entry the number $d_j$ of elements in $\eta_j$. Writing this out, we get
%\begin{align}\nonumber
%E\hat{V}_i(\Xi) &= a^{i-2}\big(c_1^{(i)} +c_2^{(i)}e^{-\gamma EV_2 (B(r))}\big)\\\label{series}
% &+ a^{i-1} c_3^{(i)}\tfrac{\gamma}{\pi} EV_1(B(r))e^{-\gamma EV_2 (B(r))}\\ \nonumber
% &+ a^i\big(c_4^{(i)} \gamma + c_5^{(i)}\big(\tfrac{\gamma}{\pi}E{V}_1(B(r))\big)^2\big)e^{-\gamma E{V}_2(B(r))} \\ \nonumber
% &+a^{i+1}\big(c_6^{(i)}\gamma E(r^{-1}) + c_7^{(i)}\tfrac{\gamma^2}{\pi}EV_1(B(r))+ c_8^{(i)}\big(\tfrac{\gamma}{\pi} EV_1(B(r))\big)^3\big)e^{-\gamma EV_2 (B(r))}\\\nonumber
%  &+ O(a^{i+2})
%\end{align} 
%where
\begin{align}\label{constants}
\begin{split}
c_1^{(i)}=&w_6^{(i)}\\
c_2^{(i)}=&w_1^{(i)}-w_6^{(i)}\\
c_3^{(i)}=&4(-w_1^{(i)} + (2-\sqrt{2})w_2^{(i)}+(-2+2\sqrt{2})w_3^{(i)} + (2-\sqrt{2})w_5^{(i)}-w_6^{(i)})\\
c_4^{(i)}=&-w_1^{(i)}+2w_2^{(i)}-2w_5^{(i)}+w_6^{(i)}\\
c_5^{(i)}=&4(2w_1^{(i)}+(-5+2\sqrt{2})w_2^{(i)}+(4-4\sqrt{2})w_3^{(i)}+(3-2\sqrt{2})w_4^{(i)}\\
&\quad +(-7+6\sqrt{2})w_5^{(i)}+(3-2\sqrt{2})w_6^{(i)})\\
c_6^{(i)}=&\frac{1}{6}(w_1^{(i)}+(2\sqrt{2}-2)w_2^{(i)}+(2-4\sqrt{2})w_3^{(i)}+(2\sqrt{2}-2)w_5^{(i)}+w_6^{(i)})\\
c_7^{(i)}=&2(2w_1^{(i)}+(-6+\sqrt{2})w_2^{(i)}+(4-2\sqrt{2})w_3^{(i)}+(2-\sqrt{2})w_4^{(i)}\\
&\quad+(-2+3\sqrt{2})w_5^{(i)}-\sqrt{2}w_6^{(i)})\\
c_8^{(i)}=&\frac{4}{3}(-8w_1^{(i)}+(22-7\sqrt{2})w_2^{(i)}+(-16+14\sqrt{2})w_3^{(i)}+(-6+3\sqrt{2})w_4^{(i)}\\&\quad +(10-13\sqrt{2})w_5^{(i)}+(-2+3\sqrt{2})w_6^{(i)}).
\end{split}
\end{align}
Note that $c_8^{(i)}=-16c_6^{(i)}-2c_7^{(i)}$.

%We now look for weights $w_j^{(i)}$ such that 
%$\lim_{a \to 0} E\hat{V}_i(\Xi)=\altoverline{V}_i(\Xi)$. 
In \cite[Theorem~9.1.4]{SW}, the following formulas for the specific intrinsic volumes, valid for the type of Boolean models we consider, are shown: 
\begin{align}\label{barV2}
\altoverline{V}_2(\Xi)&=1-e^{-\gamma E{V}_2(K)},\\\label{barV1}
\altoverline{V}_1(\Xi)&=\gamma E{V}_1(K)e^{-\gamma E{V}_2(K)},\\
\altoverline{V}_0(\Xi)&=\big(\gamma - \tfrac{1}{\pi}(\gamma E{V}_1(K)^2)\big)e^{-\gamma E{V}_2(K)}.\label{barV0}
\end{align}
These are truncated expressions of the form \eqref{pexp} with fixed constants $c_m$, so the bias of $E\hat{V}_i(\Xi)$ can be found by comparing coefficients. 

First consider $\altoverline{V}_2(\Xi)$. From \eqref{pexp} we see that
\begin{equation*}
\lim_{a\to 0} E\hat{V}_2(\Xi) = c_1^{(2)} +c_2^{(2)}e^{-\gamma EV_2 (B(r))},
\end{equation*} 
so by \eqref{barV2}, we get an asymptotically unbiased estimator for $\altoverline{V}_2(\Xi)$ exactly if $c_1^{(2)}=1$ and $c_2^{(2)}=-1$. By Equation \eqref{constants}, this means:
\begin{prop}
$\hat{V}_2(\Xi)$ is  asymptotically unbiased if and only if the weights satisfy $w_1^{(2)} =0$ and $w_6^{(2)}=1$.
\end{prop}
It is well known that $\hat{V}_2(\Xi)$ is unbiased, even in finite resolution, with the choice $w^{(2)}=\left(0,\frac{1}{4}, \frac{1}{2}, \frac{1}{2},\frac{3}{4},1\right)$, which is the estimator that counts the number of lattice points in $X$, see e.g.\ \cite[Section 4.1.1]{OM}.

Next we compare $E\hat{V}_1(\Xi)$, with \eqref{barV1} and obtain: 
%This shows that we must have $v^{(1)}=(0,0,\dots )$ in order to get an asymptotically convergent result. Comparing with \eqref{barV1}, we see that we get an asymptotically unbiased estimator exactly if $v^{(1)}=(0,0,\pi,\dots)$. One could proceed and  solve the system $v^{(1)}=(0,0,\pi,0,0, \dots)$ in order to get an estimate that converges as $o(a^2)$. Note however, that the system of linear equations $v^{(1)}=(0,0,\pi,0,0,0,0,0)$ is overdetermined. In particular, it is not possible to obtain an unbiased estimator for positive grid distance $a$, or even one that converges as $o(a^3)$. We summarize: 
\begin{thm}\label{w1}
The limit  $\lim_{a\to 0} E\hat{V}_1(\Xi)$ exists if and only if 
\begin{align}\label{w10}
w_1^{(1)}=w_6^{(1)}=0.
\end{align}
In this case, 
\begin{equation*}
\lim_{a\to 0} E\hat{V}_1(\Xi) =\tfrac{1}{\pi} c_3^{(1)}\altoverline{V}_1(\Xi).
\end{equation*}
In particular, $E\hat{V}_1(\Xi)$ is asymptotically unbiased if and only if the weights satisfy
\begin{equation}\label{w11}
c_3^{(1)}=4((2-\sqrt{2})w_2^{(1)} + (-2+2\sqrt{2})w_3^{(1)} + (2-\sqrt{2})w_5^{(1)} )= {\pi}.
\end{equation}
The bias is 
\begin{equation*}
a\big(c_4^{(1)} \gamma + c_5^{(1)}\big(\tfrac{\gamma}{\pi}E{V}_1(B(r))\big)^2\big)e^{-\gamma E{V}_2(B(r))}+O(a^2),
\end{equation*}
so the estimator converges as $O(a^2)$ exactly if the weights satisfy:
\begin{gather}\label{w12}
w_2^{(1)} - w_5^{(1)} = 0,\\ 
(-5+2\sqrt{2})w_2^{(1)} + (4-4\sqrt{2})w_3^{(1)} + (3-2\sqrt{2})w_4^{(1)} + (-7+6\sqrt{2})w_5^{(1)}=0. \label{w13}
\end{gather}
If these equations are satisfied, the bias is
\begin{equation}\label{bias1}
a^{2}\big(c_6^{(1)}\gamma E(r^{-1}) + c_7^{(1)}\tfrac{\gamma^2}{\pi}EV_1(B(r))+ c_8^{(1)}\big(\tfrac{\gamma}{\pi} EV_1(B(r))\big)^3\big) + O(a^{3}).
\end{equation}
\end{thm}

The first condition \eqref{w10} is intuitive, since lattice cells of type $\eta_1$ and $\eta_6$ will typically not contain any boundary points. 
Equation \eqref{w12} is also natural since it is exactly the condition \eqref{item2}, saying that interchanging foreground and background should not change the estimate. Equation \eqref{w11} is not so obvious. The coefficient in front of $w_j^{(1)}$ in $\frac{1}{8}c_3^{(1)}$ is the asymptotic probability that a lattice square containing a piece of the boundary is of type $\eta_j$. Equation \eqref{w13} does not seem to have a simple geometric interpretation. While \eqref{w11} and \eqref{w12} generalize to the design based setting, see Section \ref{bound} and \ref{euler}, \eqref{w13} seems to be special for the Boolean model and the underlying distribution.

The equations \eqref{w10}, \eqref{w11}, \eqref{w12}, and \eqref{w13} do not determine the weights uniquely. There is still one degree of freedom in the choice. However, this is not enough to remove the $a^2$-term in \eqref{bias1}, since the system of linear equations the weights must satisfy becomes overdetermined. The following proposition gives the best possible choice of weights:
\begin{prop}\label{opt1}
The complete solution to the system of linear equations \eqref{w10}, \eqref{w11}, \eqref{w12}, and \eqref{w13} is
\begin{equation*}
w^{(1)} = \tfrac{\pi}{16}(0,1+\sqrt{2},\sqrt{2},12+8\sqrt{2},1+\sqrt{2},0)
+w(0,1,-\sqrt{2},-4-4\sqrt{2},1,0)
\end{equation*}
where $w\in \R$ is arbitrary. 
\end{prop}
In general, the best choice of $w$ depends on the intensity $\gamma$ and the grain distribution $\Q$. Note that negative weights are allowed, even though this does not have an intuitive geometric interpretation.

Finally for the Euler characteristic, comparing $E\hat{V}_0(\Xi)$ with \eqref{barV0} yields:
\begin{thm}\label{w2}
The limit $\lim_{a\to 0} E\hat{V}_0(\Xi)$ exists if and only if
\begin{gather}
\label{w20}
w_1^{(0)}=w_6^{(0)}=0,\\\label{w21}
(2-\sqrt{2})w_2^{(0)} + (-2+2\sqrt{2})w_3^{(0)} + (2-\sqrt{2})w_5^{(0)} = 0.
\end{gather}
In this case, 
\begin{equation*}
\lim_{a\to 0} E\hat{V}_0(\Xi) = \big(c_4^{(0)} \gamma + c_5^{(0)}\big(\tfrac{\gamma}{\pi}E{V}_1(B(r))\big)^2\big)e^{-\gamma E{V}_2(B(r))}
\end{equation*}
so $\hat{V}_0$ is asymptotically unbiased if and only if the following two equations are satisfied
\begin{gather}\label{w22}
{2}w_2^{(0)} - {2}w_5^{(0)} = 1,\\ \label{w23}
(-5+2\sqrt{2})w_2^{(0)} + (4-4\sqrt{2})w_3^{(0)} + (3-2\sqrt{2})w_4^{(0)} + (-7+6\sqrt{2})w_5^{(0)} = -\tfrac{\pi}{4}.
\end{gather}
If these equations are satisfied, the bias is
\begin{equation}\label{bias2}
a\big(c_6^{(0)}\gamma E(r^{-1}) + c_7^{(0)}\tfrac{\gamma^2}{\pi}EV_1(B(r))+ c_8^{(0)}\big(\tfrac{\gamma}{\pi} EV_1(B(r))\big)^3\big) + O(a^{2}).
\end{equation}
\end{thm}
Thus the best possible weights are given by:  
\begin{prop}\label{opt2}
The general solution to the linear equations \eqref{w20}, \eqref{w21}, \eqref{w22}, and \eqref{w23} is
\begin{equation*}
w^{(0)}= \big(0,\tfrac{1}{2},-\tfrac{1}{2\sqrt{2}},\big(\tfrac{3}{4}+\tfrac{1}{\sqrt{2}}\big)(2-\pi),0,0\big)+w\big(0,1,-\sqrt{2},-4-4\sqrt{2},1,0\big)
\end{equation*}
with $w\in \R$ arbitrary. 
\end{prop}
Also here there is one degree of freedom in the choice of weights, which is not enough to annihilate the leading term of  \eqref{bias2}. 

Again the  equations \eqref{w20}, \eqref{w21}, and \eqref{w22} are geometric in the sense that they also show up in the design based setting, while  \eqref{w23} seems to be special for the Boolean model.

Note that $\hat{V}_0$ does not satisfy \eqref{item2}, not even asymptotically. For weights satisfying \eqref{w20},
\begin{align*}
\hat{V}_0(\Xi){}&=w_2^{(0)}N_2(\Xi)+w_3^{(0)}N_3(\Xi)+w_4^{(0)}N_4(\Xi)+w_5^{(0)}N_5(\Xi)\\
%\end{align*}
%Since $N_j(\Xi)=N_{7-j}({\R^2\backslash \Xi})$ for $j=2,5$, while $N_j(\Xi)=N_j({\R^2\backslash \Xi})$ for $j=3,4$,
%\begin{align*}
\hat{V}_0({\R^2\backslash \Xi}){}& =w_2^{(0)}N_5(\Xi)+w_3^{(0)}N_3(\Xi)+w_4^{(0)}N_4(\Xi)+w_5^{(0)}N_2(\Xi).
\end{align*}
Under the condition \eqref{item2}, we would thus have
\begin{align*}
2\altoverline{V}_0(\Xi)&=\lim_{a \to 0}(E\hat{V}_0(\Xi)-E\hat{V}_0({\R^2\backslash \Xi}))\\ &= \lim_{a \to 0}a^{-2}(w_2^{(0)}-w_5^{(0)})E(N_2-N_5) \\ &= (w_2^{(0)}-w_5^{(0)})\big(4\gamma +4(2-4\sqrt{2})\big(\tfrac{\gamma}{\pi }EV_1(B(r))\big)^2\big)e^{-\gamma EV_2(B(r))}
\end{align*}
which no choice of weights can satisfy by \eqref{barV0}.

Equation \eqref{w13} and \eqref{w23} become more important compared to Equation~\eqref{w12} and~\eqref{w22} when $r$ and $\gamma $ are large. These are the only equations involving the con\-fi\-gu\-ra\-tion $\eta_4$, which can only occur where two different balls are close.

\section{General Boolean models}\label{general}
The case where the grains are random balls generalizes to Boolean models where the isotropic grain distribution  satisfies the following extra condition: there is an $\eps> 0$ such that for almost all grains $K$, $B(\eps)$ slides freely inside $K$, i.e. 
%. This means that for every $x\in \partial K$ there is a ball of radius $\eps $ contained in $K$ and containing $x$. More formally, the condition is that for almost all $K$:
\begin{equation}\label{curvbound} 
\forall x\in \partial K: x-\eps n(x) + B(\eps) \subseteq K.
\end{equation}
Here $n(x)$ denotes the (necessarily unique) outward pointing unit normal vector at $x$.
% The condition %\eqref{curvbound} is equivalent to $B(\eps)$ being a summand of $K$, i.e.\ there is a convex set $L$ s.t. $K=L\oplus B(\eps)$, see \cite[Theorem 3.2.2]{schneider}
Condition~\eqref{curvbound} is a generalization of the assumption $r\geq \eps$ a.\ s.\  in Section \ref{unbiased}.

First we need a version of Equation \eqref{power} for grains satisfying \eqref{curvbound}.
In the following, $[x,y]$ denotes the closed line segment between $x,y \in \R^2$. 
\begin{lem}\label{convS}
Let $S$ be a finite set with diameter $\diam S  \leq 2\eps$. Let $K$ be a convex set satisfying~\eqref{curvbound}. Then
\begin{equation*}
V_2(\conv S \oplus K ) - V_2(S\oplus K) \leq V_2(\conv S \oplus B(\eps )) - V_2(S \oplus B (\eps ) ).
\end{equation*}
\end{lem}

\begin{proof}
After a translation, we may assume $B(\eps) \subseteq K$. Hence
\begin{equation*}
\conv S \subseteq S\oplus B(\eps) \subseteq S\oplus K.
\end{equation*}
%which implies
%\begin{equation*}
%\conv S \oplus C \backslash S \oplus C =( \conv S \oplus C \backslash \conv S) \backslash S \oplus C.
%\end{equation*}
Let $F_i$, $i\in I$, be the faces of $\conv S$ with outward pointing normal vectors $u_i$.  Then
\begin{equation}\label{conveq}
( \conv S\oplus K ) \backslash (S \oplus K )=(\conv S\oplus K ) \cap (\conv S)^c\backslash (S \oplus K) 
= \bigcup_{i\in I} ({F_i} \oplus K_{u_i}^+) \backslash (S  \oplus K)
\end{equation}
where $K_u^+ = \{z\in K \mid \langle z,u\rangle\geq 0\}$. To show the inclusion $\subseteq $ in the second equality, suppose $s\in \conv S$ and $c \in K$ with $s+c\notin \conv S$. Then there is a maximal $\lambda \in [0,1)$ such that    $s+\lambda c = f$ where $f\in \partial \conv S$. But if $f\in F_i\backslash S$, then $\langle c,u_i\rangle \geq 0$ and hence $s+c = f + (1-\lambda)c$ belongs to ${F_i} \oplus K_{u_i}^+$. If $f\in S$, then $s+c \in S\oplus C$.

Let $F_i$ be given and write $u=u_i$. After a translation we may assume $F_i=[0,x]$ with $ x\in B(2\eps)$. Let 
\begin{equation*}
y \in ([0,x] \oplus K_u^+) \backslash  (S\oplus K).
\end{equation*} 
Let  $l_y= y +\text{span}\{x\}$ be the line parallel to $[0,x]$ containing $y$. Since $K_u^+$ is convex and $y-\lambda x\in K_u^+$ for some $\lambda \in (0,1)$, $l_y\cap K_u^+ $ is a non-empty line segment $[c_1, c_2]$. Then
\begin{align}
\begin{split}\label{c1}
y&\in l_{y}\cap (  [0,x]\oplus K_u^+ )=[c_1,x+c_2]\\ 
y&\notin l_y\cap (  \{0,x\}\oplus K_u^+)= [c_1,c_2]\cup [c_1+x,c_2+x].
\end{split}
\end{align}
Choose $z\in K_u^+$ such that $n(z)=u$ and let $w=z-\eps u\in K_u^+$ be the center of the touching ball guaranteed by \eqref{curvbound}. 

By convexity, $[0,w] \oplus B(\eps) \subseteq C$, so  $l_y \cap [0,w]\neq \emptyset$ would imply
\begin{equation*}
|c_1-c_2|\geq 2\eps \geq |x|,
\end{equation*}
contradicting \eqref{c1}.  Thus
$\langle w,u \rangle \leq \langle y,u \rangle \leq \langle z,u \rangle$ and hence 
\begin{equation*}
\emptyset \neq l_y \cap [w,z]\subseteq l_y \cap (w+B(\eps)_u^+) \subseteq [c_1,c_2],
\end{equation*}
 showing that
\begin{align*}
y &\in ([0,x]\oplus (w +B(\eps)_u^+)) \backslash  (S\oplus K)\\
&\subseteq  ([0,x]\oplus (w +B(\eps)_u^+)) \backslash  (S \oplus (w + B(\eps))).
\end{align*}

Thus we may compute
\begin{align*}
V_2((\conv S\oplus K)\backslash (S\oplus K)) &\leq \sum_{i\in I} V_2( (F_i \oplus K_{u_i}^+)\backslash (S \oplus K))\\
&\leq \sum_{i\in I} V_2(( F_i \oplus B(\eps)_{u_i}^+)\backslash (S \oplus B(\eps)))\\
&=V_2((\conv S\oplus B(\eps)) \backslash (S\oplus B(\eps)))
\end{align*}
where the last equality uses the fact that when $K=B(\eps)$, the union in \eqref{conveq} is disjoint, since 
\begin{equation*}
({F_i} \oplus B(\eps)_{u_i}^+ ) \backslash (S  \oplus B(\eps)) \subseteq F_i \oplus [0,\eps u_i].
\end{equation*}
\end{proof}

Now let $\xi_l$ be a configuration and write $F_l=\conv (\xi_l)$. Then Lemma \ref{convS} implies:
\begin{cor}
Let $\Xi$ be a Boolean model such that for some $\eps>0$, the grains satisfy~\eqref{curvbound} almost surely. For $\sqrt{2}a<\eps$ and $l=0,\dots , 15$,
\begin{equation*}
EV_2(F_l\oplus K)-EV_2(\xi_l \oplus K) \leq a^3\eps^{-1}V_1^{(3)}(a^{-1}\xi_l)+O(a^5).
%\leq V_2(F_l\oplus B_\eps)-V_2(\xi_l \oplus B_\eps ) =O(a^3). 
\end{equation*}
\end{cor}

This allows us to compute $P(\xi_l \subseteq \R^2 \backslash \Xi)$ using  \eqref{Pzeta} and \eqref{steiner}, but only up to second order:
\begin{align}\label{PC}
P(\xi_l {}&\subseteq \R^2 \backslash \Xi) =e^{-\gamma (EV_2( K) + a\frac{2}{\pi}V_1(F_l)EV_1(K) + a^2 V_2(F_l) + O(a^3))}\\ \nonumber
&=c_1+e^{-\gamma EV_2(K)}\big(c_2 + ac_3 \tfrac{\gamma}{\pi} EV_1(K)  + a^2\big(c_4\gamma +c_5\big(\tfrac{\gamma}{\pi}EV_1(K)\big)^2\big)\big) +O(a^3)
\end{align}
with the same constants $c_m$ as in Section \ref{unbiased}, since these depend only on $V_i(a^{-1}F_l)$.

Furthermore, the specific intrinsic volumes were given by \eqref{barV2}--\eqref{barV0}
%\begin{align*}
%\altoverline{V}_2(\Xi)&=1-e^{-\gamma E{V}_2(C)},\\
%\altoverline{V}_1(\Xi)&=\gamma E{V}_1(C)e^{-\gamma E{V}_2(C)},\\
%\altoverline{V}_0(\Xi)&=\left(\gamma - \frac{1}{\pi}(\gamma E{V}_1(C))^2\right)e^{-\gamma E{V}_2(C)},
%\end{align*}
so by exactly the same arguments as in Section \ref{unbiased}, we find:
\begin{thm}
Theorem \ref{w1} and \ref{w2}, except for Equation \eqref{bias1} and \eqref{bias2}, also hold for an isotropic Boolean model with grains satisfying \eqref{curvbound} almost surely. 
\end{thm}

\begin{rem}
The term $O(a^3)$ in \eqref{PC} is of the form
\begin{equation*}
a^3\big(c_7\tfrac{\gamma^2}{\pi}EV_1(K)+ c_8\big(\tfrac{\gamma}{\pi} EV_1(K)\big)^3\big)+ \gamma \phi(a)+ O(a^4)
\end{equation*}
where $c_7$ and $c_8$ are as in \eqref{pexp}, and $0\leq \phi(a) \leq c_6 \eps^{-1}a^3$ with $c_6$ as in \eqref{pexp}.
\end{rem}

\section{Generalization to standard random sets}\label{standard}
As an easy consequence of well-known results obtained in \cite{eva}, the first-order results for Boolean models generalize further to isotropic standard random sets. 
A standard random set $Z$ is a stationary random closed set, such that the realizations are a.\ a.\ locally polyconvex and $Z$ satisfies the integrability condition 
\begin{equation*}\label{integrability}
E2^{N(Z\cap B(1))} < \infty
\end{equation*}
where $N(Z\cap B(1))$ is the minimal number $n$ such that $Z\cap B(1)$ is a union of $n$ convex sets, see also \cite[Definition 9.2.1]{SW}. 

The specific intrinsic volumes of a standard random set are defined as in \eqref{defVi} and we estimate $\altoverline{V}_1$ by
\begin{equation*}
\hat{V}_1(Z) = a^{-1}\sum_{j=1}^6 w_j^{(1)}  \frac{N_j}{N_0}
\end{equation*} 
as in \eqref{Nest} where $N_j$ are as in \eqref{indik}. 
Since lower dimensional parts of $Z$ are usually invisible in the digitization, we assume that $Z$ is a.\ s.\ topologically regular.%, i.e. it coincides with the closure of its interior.
%\begin{equation*}
%N_j = \sum_{z \in a\Z^2\cap (W \ominus a\check{C})} 1_{(Z-z) \cap aC_0 \in \eta_j}.
%\end{equation*}

\begin{thm}\label{stand}
Let $Z$ be an isotropic standard random set in the plane which is a.\ s.\ topologically regular. Then $\lim_{a\to 0} E\hat{V}_1(Z)$ exists if and only if $w_1^{(1)}=w_6^{(1)}$. In this case,
\begin{equation*}
\lim_{a\to 0} E\hat{V}_1(Z) =\tfrac{1}{\pi}c_3^{(1)} \altoverline{V}_1(Z)
\end{equation*}
with $c_3^{(1)}$ as in \eqref{constants}. In particular, $\hat{V}_1(Z)$ is asymptotically unbiased exactly if \eqref{w11} holds.
\end{thm}

%A set is topologically regular if it is the closure of its interior.
\begin{proof}
As in the case of the Boolean model,
\begin{equation*}
E\hat{V}_1(Z) = a^{-1}\sum_{j=1}^6 w_j^{(1)}  P(Z\cap aC_0\in \eta_j).
\end{equation*} 

First let $\xi_l$, $l\neq 0,15$, be a configuration with  $B_l,W_l\neq \emptyset$. Define the support function of a set $A$ by
$h(A,n)= \su \{\langle x,n \rangle \mid x \in A \}$
for $n\in S^1$ and $\langle \cdot,\cdot \rangle$ the standard Euclidean inner product. 
The following formula is shown in \cite[Theorem 4]{eva}:
\begin{equation*}
\lim_{a\to 0}a^{-1} P(B_l\subseteq Z,W_l\subseteq Z^c) = \int_{S^1}(-h(B_l\oplus \check{W_l}),n)^+\bar{L}(dn).
\end{equation*}
Here $x^+=\textrm{max}\{x,0\}$ and $\bar{L}$ is the mean normal measure on $S^1$
\begin{equation*}
\bar{L}(A) = \lim_{r\to \infty}\frac{ES_1(Z\cap B(r);A)}{V_2(B(r))},\text{ } A\in \mathcal{B}(S^1),
\end{equation*}
where $S_1(K;\cdot)$ is the first area measure, see \cite[Chapter 4]{schneider} when $K$ is polyconvex.
In particular, the total measure $\bar{L}(S^1)$ is $2\altoverline{V}_1(Z)$.

By the isotropy of $Z$, $\bar{L}$ is rotation invariant, so Tonelli's theorem yields
\begin{align*}
\lim_{a\to 0}a^{-1} P(B_l\subseteq Z,W_l\subseteq Z^c) 
%&=\lim_{a\to 0}a^{-1} P(R_vB_l\subseteq Z,R_vW_l\subseteq Z^c)\\
&= \int_{S^1}(-h(B_l\oplus \check{W_l},n))^+\bar{L}(dn)\\
&=\frac{1}{2\pi}\int_0^{2\pi}\int_{S^1}(-h(B_l\oplus \check{W_l},R_{-v}n))^+\bar{L}(dn)dv\\
&=\frac{1}{2\pi}\int_{S^1}\int_0^{2\pi}(-h(B_l\oplus \check{W_l},u_v))^+dvd\bar{L}
\end{align*}
where $u_v=(\cos v,\sin v)$.
The inner integral depends only on the equivalence class $\eta_j$ containing $\xi_l$. Thus we only need to compute it for one representative $\xi_{l_j}$ of each $\eta_j$.
%\begin{align*}
%(-h(B_1\oplus \check{W}_1,u_v))^+{}&=(-h(B_7\oplus \check{W}_7,u_v))^+=\begin{cases}
%\max\{|\cos v|,|\sin v|\},&v\in [0,\frac{\pi}{2}],\\
%%\sin v,&v\in [\frac{\pi}{4},\frac{\pi}{2}],\\
%0,&\textrm{otherwise}.
%\end{cases}\\
%(-h(B_3\oplus \check{W}_3,u_v))^+{}&=\begin{cases}
%\max\{|\cos v|,|\sin v|\}-\min\{|\cos v|,|\sin v|\},& v\in [\frac{\pi}{4},\frac{3\pi}{4}],\\
%%\cos v +\sin v ,& v\in [\frac{\pi}{2},\frac{3\pi}{4}],\\
%0,&\textrm{otherwise}.
%\end{cases}\\
%(-h(B_6\oplus \check{W}_6,u_v))^+{}&=0.
%\end{align*}
\begin{align*}
(-h(B_1\oplus \check{W}_1,u_v))^+{}&=(-h(B_7\oplus \check{W}_7,u_v))^+=
\max\{|\cos v|,|\sin v|\}1_{v\in [0,\frac{\pi}{2}]}
\\
(-h(B_3\oplus \check{W}_3,u_v))^+{}&=(
\max\{|\cos v|,|\sin v|\}-\min\{|\cos v|,|\sin v|\})1_{ v\in [\frac{\pi}{4},\frac{3\pi}{4}]}\\
%\cos v +\sin v ,& v\in [\frac{\pi}{2},\frac{3\pi}{4}],\\
(-h(B_6\oplus \check{W}_6,u_v))^+{}&=0.
\end{align*}
A direct computation now shows that
\begin{align*}
\lim_{a\to 0}a\sum_{j=2}^5 w_j^{(1)}EN_j &= \sum_{j=2}^5 w_j^{(1)}d_j \frac{1}{2\pi}\int_{S^1}\int_0^{2\pi} (-h(B_{l_j}\oplus \check{W}_{l_j},u_v))^+ dv d\bar{L}
 %&= \frac{1}{2\pi}\int_{S^1}(w_2^{(1)} 4(2-\sqrt{2}) + w_3^{(1)} 4(-2+2\sqrt{2}) + w_5^{(1)} 4(2-\sqrt{2}))d\bar{L}\\ 
 = \frac{1}{\pi}c_3^{(1)}\altoverline{V}_1(Z).
\end{align*}

Finally, it is well-known that 
\begin{align*}
\lim_{a\to0}P(Z\cap aC_0\in \eta_6)&=\altoverline{V}_2(Z),\\
\lim_{a\to0}P(Z\cap aC_0\in \eta_1)&=1-\altoverline{V}_2(Z),
\end{align*}
so we must choose $w_1^{(1)}=w_6^{(1)}=0$ in order for $\lim_{a\to0}E\hat{V}_1(Z)$ to exist for all $Z$.
\end{proof}

\section{Boundary length in the design based setting}\label{bound}
Instead of considering random sets observed on a fixed lattice, we now turn to the design based setting where we sample a deterministic compact set $X\subseteq \R^2$ with a stationary isotropic random lattice, by which we mean that $\La$ is the random set $\La(c,v) = R_v(\Z^2 + c)$ where $v \in [0,2\pi]$ and $c\in C$ are mutually independent uniform random variables.

We first consider estimators for the boundary length $2V_1$, as this is a fairly easy consequence of \cite[Theorem 5]{rataj}. 
Based on the random set $X\cap a\La $, we consider an estimator of the form 
\begin{equation*}
\hat{V}_1(X) = a\sum_{j=1}^6 w_j^{(1)}  N_j(  X \cap a\La),
\end{equation*}
as described in Section \ref{notation} and study the asymptotic behavior of $E\hat{V}_1(X)$. 

%Again we look for conditions for $\hat{V}_1$ to be asymptotically unbiased, i.e.\
%\begin{equation}%\label{unb}
%\lim_{a \to 0} E\hat{V}_1(X)=V_1(X).
%\end{equation}

We first need some conditions on $X$.
A compact set $X\subseteq \R^2$ is called gentle, see \cite{rataj}, if the following two conditions hold:
\begin{itemize}
\item[(i)] $\Ha^1(\mathcal{N}(\partial X)) < \infty$,
\item[(ii)] For $\Ha^{1}$-almost all $x\in \partial X$, there exist two balls $B_i$ and $B_o$ with non-empty interior, both containing $x$, and such that $B_i\subseteq X$ and $\indre(B_o)\subseteq \R^2\backslash X$.
\end{itemize}
Here and in the following $\Ha^d$ denotes the $d$-dimensional Hausdorff measure, and $\mathcal{N}(\partial X)$ is the reduced normal bundle
\begin{equation*}
\mathcal{N}(\partial X)=\{(x,n)\in \partial X\times S^1\mid \exists t>0:\forall y\in \partial X :|tn|<|tn+x-y|\}.
\end{equation*}
%The last condition means that $(x,n)\in \mathcal{N}(\partial X)$ if there is a $t>0$ such that $x$ is the point in $\partial X$ closest to $x+tn$.

\begin{thm}\label{boundary}
Let $X\subseteq \R^2$ be a  compact gentle set and $\La$ a stationary isotropic random lattice. Then $\lim_{a \to 0} E\hat{V}_1(X)$ exists iff $ w_6^{(1)}=w_1^{(1)}=0$.
In this case,
\begin{equation*}
\lim_{a \to 0} E\hat{V}_1(X)=\tfrac{1}{\pi}c_3^{(1)}V_1(X)
\end{equation*}
with $c_3^{(1)}$ as in \eqref{constants}.
In particular, $\hat{V}_1(X)$ is asymptotically unbiased if and only if $w^{(1)}$ satisfies Equation~\eqref{w11}.  
\end{thm}

In Section \ref{euler} we shall see that under stronger conditions on $X$, the convergence is actually $O(a)$ and the weights can be chosen so that it is even $O(a^2)$.

Theorem 5 of \cite{rataj} is only shown for a uniformly translated lattice, whereas we assume isotropy as well. Thus we need the following lemma. 

%\begin{lem}\label{polyconvex}
%There is a constant $M$ such that for  any polyconvex set $X$ and any lattice $\La$ with unit grid distance, 
%\begin{equation*}
%N_j(X\cap a\La )\leq a^{-1} M  N(X)(V_1(X)+a).
%\end{equation*}
%\end{lem} 
%\begin{proof}
%If $\La=\La(c,v)$, let $N_\partial (X)$ denote the number of $z\in a\La$ such that $(z+aR_v(C))\cap \partial X \neq \emptyset$, i.e.\ the number of lattice cells that contain boundary points. 
%
%First let $K$ be convex. There is a sequence of polygons $P_i \subseteq K \subseteq P_i+ B_{\eps_i}$ with $\lim_{i\to \infty} \eps_i = 0$.  A polygon satisfies 
%\begin{equation*}
%N_\partial(P_i)\leq 4a^{-1}(\Ha^1(\partial P_i)+a)
%\end{equation*}
%since a piecewise linear curve of length $a$ can run through at most 4 lattice cells.
%Then if $\eps_i\leq a$,
%\begin{equation*}
%N_\partial (K)\leq 9N_\partial(P_i)\leq 36a^{-1}(\Ha^1(\partial P_i)+a)
%\end{equation*}
%since any boundary point of $K$ has distance at most $a$ to a boundary point of $P_i$, hence it must lie in either the same or one of the 8 neighboring lattice cells.
%Letting $i\to \infty$,  
%\begin{equation*}
%N_\partial(K)\leq 36 a^{-1}(\Ha^1(\partial K)+a).
%\end{equation*}

%
% Write $X = K_1\cup\dots \cup K_{N(X)}$ where the $K_i$ are compact convex sets. Then 
%\begin{align*}
%N_j(\La \cap X\cap B)&\leq \sum_{i=1}^{N(X)} N_\partial(K_i\cap \La ) \\
%&\leq \sum_{i=1}^{N(X)} 36a^{-1}(\Ha^1(\partial K_i)+a) \\
%&\leq N(X)\cdot 36a^{-1}(\Ha^1(\partial X)+a).
%\end{align*}
%\end{proof}

\begin{lem}\label{polyconvex}
For any compact gentle set $X$ there is an $\eps >0$ such that for any square lattice $\La$ with unit grid distance,
\begin{equation*}
N_j(X\cap a\La)\leq a^{-1}( 1 + 4\sqrt{2}V_1(X))
\end{equation*}
for all $a<\eps$ and $j=2,\dots ,5$.
\end{lem}

\begin{proof}
If $(z+aR_vC_0)\cap \partial X$ is a configuration of type $j\neq 1,6$ for some $z\in a\La$, then
%Let $N_\partial(X\cap a\La)$ be the number of $z\in a\La$ such that
 $(z+aR_vC)\cap \partial X\neq \emptyset$ and hence
\begin{equation*}
z+aR_vC\subseteq \partial X\oplus B\big(\sqrt{2}a\big).
\end{equation*}
Thus
\begin{equation*}
%\begin{split}
N_j(X\cap a\La)
%&\leq N_\partial(X\cap a\La)\\
\leq a^{-2}V_2\big(\partial X\oplus B\big(\sqrt{2}a\big)\big).
%&=a^{-2}\int_{\R^2}1_{ \partial X\oplus B(\sqrt{2}a)  }d\Ha^2.
%\end{split}
\end{equation*}
%The second inequality holds because $(z+aR_vC)\cap \partial X\neq \emptyset$ implies that 

\cite[Theorem 1]{rataj} with $P=B(\sqrt{2}a)$ and $Q=B(ar )$ shows that
\begin{equation*}
\lim_{a\to 0} a^{-1} V_2\big(X\oplus B\big(\sqrt{2}a\big)\backslash X \ominus B(ar)\big) =\big(\sqrt{2}+r\big)2V_1(\partial X).
\end{equation*}
Letting $r=\sqrt{2}\pm\eps$ for $\eps \to 0$ yields
\begin{equation*}
\lim_{a\to 0} a^{-1} V_2\big(\partial X \oplus B\big(\sqrt{2}a\big)\big) =4\sqrt{2}V_1( X).
\end{equation*}
In particular, $a^{-1} V_2\big(\partial X \oplus B\big(\sqrt{2}a\big)\big) - 4\sqrt{2}V_1( X) \leq 1$ for all $a$ sufficiently small.
\end{proof}

\begin{proof}[Proof of Theorem \ref{boundary}]
Since $X$ is compact, $N_1$ is infinite, so $w_1^{(1)}$ must equal zero in order for the estimator to be well-defined. Moreover, $\lim_{a\to 0}a^2N_6=V_2(X)$. Thus $aN_6$ diverges when $a\to 0$, while all other $aN_j$ remain bounded  by Lemma \ref{polyconvex}. Hence $w_6^{(1)}=0$ is necessary for $\lim_{a\to 0}E\hat{V}_1(X)$ to exist.

%Let $\xi_l$ be a configuration with $l\neq 0,15$. Theorem 5 of \cite{rataj} then reads:
%\begin{equation*}
%\lim_{a\to 0} a \int_C N_l(X\cap a\La(v,c)) dc = \int_{S^1} (-h(R_v(B_l)\oplus R_v(\check{W}_l),n))^+ S_1(X;dn).
%\end{equation*} 
%where $S_1(X; \cdot)$ again denotes the first area measure on $S^1$.

%We must compute
%\begin{equation*}
%\lim_{a\to 0} a E N_l(X\cap a\La) = \lim_{a\to 0} a \frac{1}{2\pi}\int_0^{2\pi}\int_C N_l(X\cap a\La(v,c)) dc dv. 
%\end{equation*} 
By Lemma \ref{polyconvex},
$aN_l(X\cap a\La(v,c))  $
is uniformly bounded, so using the Lebesgue theorem of dominated convergence:
\begin{align*}
\lim_{a\to 0} a E N_l(X\cap a\La(v,c)) =&\lim_{a\to 0} a \frac{1}{2\pi}\int_0^{2\pi}\int_C N_l(X\cap a\La(v,c)) dc dv\\
=&\frac{1}{2\pi}\int_0^{2\pi} \lim_{a\to 0} a \int_C N_l(X\cap a\La(v,c)) dc dv \\=& \frac{1}{2\pi}\int_{S^1}\int_0^{2\pi} (-h(R_v(B_l)\oplus R_v(\check{W}_l),n))^+ dv S_1(X;dn)\\=& \frac{1}{2\pi}\int_{S^1}\int_0^{2\pi} (-h(B_l\oplus \check{W}_l,R_{-v}n))^+ dv S_1(X;dn).
%\\=& \frac{1}{2\pi}\int_{S^1}\int_0^{2\pi} (-h(B_l\oplus \check{W}_l,u_v))^+ dv S_1(X;dn).
\end{align*} 
where the third equality is Theorem 5 of \cite{rataj}.
The remaining computations are as in the proof of Theorem \ref{stand}, since $S_1(X;S^1)=2V_1(X)$.
\end{proof}
 
Note how the isotropy of the lattice was crucial in the proof. This corresponds to the isotropy requirement for the Boolean model.

\section{Euler characteristic in the design based setting}\label{euler}
We remain in the design based setting of Section \ref{bound} and consider the estimation of the Euler characteristic and the higher order behavior of boundary length estimators. For this, we need some stronger boundary conditions on $X$. For instance, J\"{u}rgen Kampf has shown in a yet unpublished paper that without the isotropy of the lattice, there are no local estimators for $V_0$ that are asymptotically unbiased for all polyconvex sets. On the other hand, it is well-known that there exists a local algorithm for $V_0$ which is asymptotically unbiased on the class of so-called $r$-regular sets, see e.g.\ the discussion in \cite{kothe}. We will assume throughout this section that $X$ is a compact full-dimensional $C^2$ manifold, which is slightly stronger than $r$-regularity.  

The estimator for the Euler characteristic was defined in Section \ref{notation} as
\begin{equation*}
\hat{V}_0(X)=\sum_{j=1}^6 w_j^{(0)} N_j(X\cap a\La).
\end{equation*}
Note that $\hat{V}_1(X)=a\hat{V}_0(X)$ if $w_j^{(1)}=w_j^{(0)}$. 
To treat both estimators, we sometimes just write $w_j^{(i)}$ for the weights.
As noted in Section \ref{bound}, we must choose $w_1^{(i)}=0$ in order for $\hat{V}_i$ to be well-defined and $w_6^{(i)}=0$ to make $a^{1-i}E\hat{V}_i(X)$ asymptotically convergent. Hence we assume $w_1^{(i)}=w_6^{(i)}=0$ throughout this section.

The main result we shall obtain is the following:
\begin{thm}\label{EC}
Assume $X\subseteq \R^2$ is a compact  2-dimensional $C^2$ submanifold with  boundary. 
Then
\begin{equation*}
\lim_{a\to 0} (E\hat{V}_0(X)- a^{-1}\lim_{a\to 0} a E\hat{V}_0(X))= c_4^{(0)}V_0(X)
\end{equation*}
with $c_4^{(0)}$ as in \eqref{constants}. 
Thus, $\lim_{a\to 0} E\hat{V}_0(X)$ exists iff the weights satisfy \eqref{w21} and $\hat{V}_0(X)$ is asymptotically unbiased iff \eqref{w22} holds. In this case, $E\hat{V}_0(X)$ satisfies~\eqref{item2} asymptotically. 

%Under the condition \eqref{w10}, $E\hat{V}_1(X)$ satisfies
%\begin{equation*}
%\lim_{a\to 0}a^{-1}(E\hat{V}_1(X) -\lim_{a\to 0}E\hat{V}_1(X)) =  c_4^{(1)}V_0(X),
%\end{equation*}
Moreover, $E\hat{V}_1(X)$ converges as $O(a)$, and if \eqref{w12} is satisfied, even as $o(a)$. In this case, $\hat{V}_1(X)$ satisfies \eqref{item1}.
\end{thm}

Theorem \ref{EC} generalizes the equations \eqref{w12} and \eqref{w22} to the design based setting. However, the equations \eqref{w13} and \eqref{w23} do not appear. These involve the configuration $\eta_4$, which cannot occur when the boundary is  $C^2$ and $a$ is sufficiently small.

For the proof, we must compute 
\begin{equation*}
\sum_{j=2}^5 w_j^{(i)} EN_j = \sum_{j=2}^5 w_j^{(i)}\frac{1}{2\pi}\int_0^{2\pi }\int_C N_j(X\cap a\La(c,v)) dc dv.
\end{equation*}
We follow the same approach as in \cite{rataj}. The idea is that 
\begin{equation*}
N_j(X\cap a\La(c,v)) =\sum_{l: \xi_l \in \eta_j} \sum_{z\in a\La(c,v)} 1_{\{z+aR_v(B_l)\subseteq X\} }1_{\{z+aR_v(W_l)\subseteq \R^2\backslash X\}}.
\end{equation*}
Integrating over all $c\in C$,
\begin{equation}
\int_C N_j(X\cap a\La(c,v)) dc  %\\ \nonumber &= 
%{a^{-2}}\sum_{l: \xi_l \in \eta_j} \Ha^2(z\in \R^2 \mid z+aR_v(B_l)\subseteq X ,z+aR_v(W_l)\subseteq \R^2\backslash X )\\
=a^{-2}\sum_{l: \xi_l \in \eta_j}\int_{\R^2}f_l(z,v)\Ha^2(dz)\label{Hmaal}
\end{equation}
where $f_l$ denotes the indicator function
\begin{equation}\label{fl} 
f_l(z,v) = 1_{\{z +a R_v(B_l) \subseteq X\}} 1_{\{z+ a R_v(W_l) \subseteq \R^2\backslash X\} }.
\end{equation}

By the assumptions on $X$, there is a  unique outward pointing normal vector $n(x) $ at $x$. 
Since  $\partial X$ is an embedded $C^2$ submanifold, the tubular neighborhood theorem ensures that there is an $\eps >0$ such that all points in $\partial X \oplus B(\eps)$ have a unique closest point in $\partial X$. 
For $ \sqrt{2} a< \eps$, the support of $f_l$ is contained in $\partial X \oplus B(\eps)$. 

As in the proof of \cite[Theorem 1]{rataj}, we apply \cite[Theorem 2.1]{last} to compute \eqref{Hmaal}. In the case of $C^2$ manifolds, this reduces to the Weyl tube formula:
\begin{align}
\begin{split}\label{intform}
\int_{\R^2} f_l(z,v)\mathcal{H}^2(dz) =\int_{\partial X}\int_{-\eps}^{\eps}tf_l(x+tn,v) &k(x)dt \mathcal{H}^{1}(dx) \\&+\int_{\partial X}\int_{-\eps}^{\eps}f_l(x+tn,v) dt\mathcal{H}^{1}(dx)
\end{split}
\end{align} 
where $k(x)$ is the signed curvature at $x$.

The main part of the proof of Theorem \ref{EC} is contained in Lemma \ref{l1} and \ref{l2}, handling each of the two integrals in \eqref{intform}. 
Before proving these, we show two technical lemmas.
The first is a standard differential geometric description of $\partial X$. 

In the following, $\tau(x)$ denotes the  unit tangent vector at $x$ chosen so that $\{\tau(x),n(x)\}$ are positively oriented.
\begin{lem}\label{lemma}
Let $X \subseteq \R^2$ be a $C^2$ submanifold with boundary.
For some $\delta <0$, there is a well-defined $C^{1}$ function $l: [-2\delta,2\delta]\times \partial X  \to \R$ such that $l(r,x)$ is the 
signed length of the line segment parallel to $n(x)$ from $x + r\tau(x)$ to $\partial X$. 
The sign is chosen such that $x + r\tau(x)+l(r,x)n(x)\in \partial X$.

The function
$r^{-2} {l(br,x)}$
is bounded for $(b,r,x) \in [-2,2] \times [-\delta,\delta]\backslash \{0\}\times \partial X$ and
\begin{align*}
\lim_{r\to 0}{r^{-2}} {l(br,x)}&= -\tfrac{1}{2}b^2k(x). 
\end{align*}
\end{lem}

\begin{proof}
By the assumptions on $X$, there are finitely many isometric $C^2$ parametrizations of the form  $\alpha : (a-2\mu,b+2\mu) \to \partial X$ such that the sets $\alpha([a,b])$ cover $\partial X$. For any $t\in (a-2\mu,b+2\mu)$,  we write $n(t)=n(\alpha(t))$ for shot.  There are unique functions $l,r : (-\mu,\mu)\times (a-\mu,b+\mu)\to \R$ such that for any $(s,t) \in  (-\mu,\mu)\times (a-\mu,b+\mu)$,
\begin{align*}
\alpha(s+t)-\alpha(t) =r(s,t)\alpha'(t) + l(s,t)n(t)
\end{align*}
where 
\begin{align*}
r(s,t)= \langle \alpha(s+t) -\alpha(t), \alpha'(t)\rangle,\\
l(s,t)= \langle \alpha(s+t)- \alpha(t), n(t)\rangle.
\end{align*}
In particular, note that both functions are $C^{1}$, and as functions of $s$ they are even $C^2$. In an open neighborhood of $[a,b]\times 0$, $\frac{\partial}{\partial s}r(s,t) > 0$. By the inverse function theorem applied to $(r(s,t),t)$, there is a $\delta$  such that the inverse $s(r,t)$ is defined and is $C^{1}$ on $(-3\delta, 3\delta) \times [a,b]$. In fact, $r\mapsto s(r,t)$ is $C^{2}$ as it is the inverse of $s\mapsto r(s,t)$. Then $l(s(r,t),t)$ is the distance from $\alpha(t) + r\alpha'(t)$ to $\alpha(s(r,t)+t) $. If $3\delta < \eps $, this is the boundary point on the line parallel to $n(t)$ closest to $\alpha(t) + r\alpha'(t)$.

 By the mean value theorem,
\begin{align}\label{begr}
\begin{split}
\frac{l(s(br,t),t)}{r} &= b\frac{\partial }{\partial s}l(s,t)\mid_{s=s(br_0,t)}\frac{\partial}{\partial r}s(r,t)\mid_{r=br_0},\\
\frac{l(s(br,t),t)}{r^2} &= b^2\frac{r_0}{r}\frac{\partial^2}{\partial s^2}l(s,t)\mid_{s=s(br_1,t)}\frac{\partial }{\partial r}s(r,t)\mid_{r=br_0}\frac{\partial }{\partial r}s(r,t)\mid_{r=br_1},
\end{split}
\end{align}
 for some $0\leq |r_1|\leq |r_0 |\leq |r|$. 
%But 
% \begin{equation*}
%\frac{\partial }{\partial r}s(r,t) = \frac{1}{\frac{\partial }{\partial s} r(s(r,t),t)}\leq 2
% \end{equation*}
The continuity of $\frac{\partial}{\partial s}l$, $\frac{\partial^2}{\partial s^2}l$ and $\frac{\partial}{\partial r}s$ on $[-2\delta, 2\delta]\times [a,b]$ implies that \eqref{begr} is bounded on $[-2,2]\times [-{\delta},{ \delta}]\backslash \{0\}\times [a,b]$.

Finally, since $l(s(0,t),t)=0$ and $\frac{\partial}{\partial s}l(s,t)\mid_{s=0}=0$, we obtain
\begin{equation*}
\begin{split}
\lim_{r\to 0}\frac{l(s(br,t),t)}{r} &=\frac{\partial}{\partial r}l(s(br,x))\mid_{r=0}=0\\
\lim_{r\to 0}\frac{l(s(br,t),t)}{r^2} &= \frac{1}{2}\frac{\partial^2}{\partial r^2}l(s(br,x))\mid_{r=0}
%%&= \frac{1}{2}\bigg(\frac{\partial^2}{\partial s^2}l(s,t)\mid_{s=0}\left(\frac{\partial}{\partial r}s(br,x)\mid_{r=0}\right)^2 \\ 
%%&\qquad \qquad \qquad \qquad \qquad \qquad +\frac{\partial}{\partial s}l(s,t)\mid_{s=0}\left(\frac{\partial^2}{\partial r^2}s(br,x)\mid_{r=0}\right)\bigg) \\
=\frac{1}{2}b^2\langle\alpha''(t),n(t)\rangle 
=-\frac{1}{2}b^2k(\alpha(t)), 
\end{split}
\end{equation*}
proving the last claim.
\end{proof}
Before proving the next lemmas, we introduce some notation. Let $v\in [0,2\pi]$ and $x\in \partial X$. Let $v_0,\dots, v_3$ be the elements of $R_v(C_0)$ ordered such that $s_i\geq s_{i+1}$ where $s_i=\langle v_i,n(x)\rangle$. Let $b_i = \langle v_i,\tau(x)\rangle$. Note that the ordering of the $v_i$ depends only on $R_{-v}n\in S^1$, and that $S^1$ is divided into 8 arcs of length $\frac{\pi}{4}$ on each of which the ordering of the $R_v(C_0)$ is constant as a function of $R_{-v} n\in S^1$. The $s_i$ and $b_i$ can be be computed explicitly as a function of $R_{-v}n \in S^1$. Though used in the explicit calculations below, these values have been omitted.

Define
\begin{equation*}
t_i=-as_i +l(b_ia,x).
\end{equation*}
The $t_i$ are constructed such that for $t\in [-\eps,\eps]$,
\begin{equation}\label{smut}
 x+tn(x)+av_i\in X \textrm{ if and only if } t\leq t_i. 
\end{equation}

%The values of $s_i$ and $b_i$ are given in Table \ref{tabel} for values of $u=\theta(n(x))-v\in (-\frac{5\pi}{4},-\frac{\pi}{4})$ where $\theta(n(x))$ is the angle between $n(x)$ and the vector $(1,0)$. For $u$ and $-u-\frac{\pi}{4}$ the value of $s_i$ is the same, while $b_i$ changes sign. This yields the values of $s_i$ and $b_i$ for $u\in (-\frac{\pi}{4},\frac{3\pi}{4})$. In the table, $w$ is chosen such that $w\in (0,\frac{\pi}{4})$.

%\begin{table}
%\begin{center}
%\begin{tabular}{llccc}
%\hline 
%{}&{$u=\theta(n)-v$}&{$i$}&{$s_i$}&{$b_i$}\\
%\hline
%{$u\in(-\frac{\pi}{2},-\frac{\pi}{4})$}&{$w=u+\frac{\pi}{2}$}&{$0$}&{$\cos w +\sin w$}&{$\cos w -\sin w$}\\
%{}&{}&{$1$}&{$\cos w $}&{$-\sin w$}\\
%{}&{}&{$2$}&{$\sin w $}&{$\cos w$}\\
%{}&{}&{$3$}&{$0$}&{$0$}\\
%\hline
%{$u\in(-\frac{3\pi}{4},-\frac{\pi}{2})$}&{$w=-u-\frac{\pi}{2}$}&{$0$}&{$\cos w $}&{$\sin w$}\\
%{}&{}&{$1$}&{$\cos w -\sin w$}&{$\cos w+\sin w$}\\
%{}&{}&{$2$}&{$0$ }&{$0$}\\
%{}&{}&{$3$}&{$-\sin w $}&{$\cos w$}\\
%\hline
%{$u\in(-{\pi},-\frac{3\pi}{4})$}&{$w=u+{\pi}$}&{$0$}&{$\sin w$}&{$\cos w $}\\
%{}&{}&{$1$}&{$0$}&{$0$}\\
%{}&{}&{$2$}&{$-(\cos w -\sin w)$ }&{$\cos w + \sin w$}\\
%{}&{}&{$3$}&{$-\cos w$}&{$\sin w$}\\
%\hline
%{$u\in(-\frac{5\pi}{4},-{\pi})$}&{$w=-u-{\pi}$}&{$0$}&{$0$}&{$0$}\\
%{}&{}&{$1$}&{$-\sin w $}&{$\cos w$}\\
%{}&{}&{$2$}&{$-\cos w $}&{$-\sin w$}\\
%{}&{}&{$3$}&{$-(\cos w + \sin w)$}&{$\cos w -\sin w$}\\
%\hline
%\end{tabular}
%\end{center}
%\caption{Values of $s_i$ and $b_i$ for $u\in (-\frac{5\pi}{4},-\frac{\pi}{4})$.} \label{tabel}
%\end{table}

Let $t_i'$ be a reordering of the  $t_i$ such that $t_i' \leq t_{i+1}'$ and let $v_i'$ be the corresponding ordering of the $v_i$. This ordering depends on both $x$, $v$ and $a$. Since $t_i$ may not equal $t_i'$, we need the following lemma, ensuring that this does not happen too often:

\begin{lem}\label{order}
There is a constant $M$ such that for all $x\in \partial X$ and $a$ sufficiently small,
\begin{equation*}
a^{-1}\Ha^1(v\in [0,2\pi]\mid \exists i: v_i\neq v_i') \leq M.
\end{equation*}
Furthermore, there is a constant $M'$ such that
\begin{equation*}
|t_i-t_i'|\leq 4\su \{|l(ba,x)| \mid (b,x)\in [-\sqrt{2},\sqrt{2}]\times \partial X\}\leq M'a^2.
\end{equation*}
\end{lem}

\begin{proof}
Let $v\in [0,2\pi]$ and $x\in \partial X$ given.
If $v_i\neq v_i'$, then in particular there is a $j_1<j_2$ with $t_{j_1}>t_{j_2}$. But then 
\begin{equation}\label{prime}
0\leq t_{j_1}-t_{j_2} = a(s_{j_2}-s_{j_1}) +l(b_{j_1}a,x)-l(b_{j_2}a,x)
\end{equation}
and hence
\begin{equation*}
0\leq  a(s_{j_1}-s_{j_2}) \leq l(b_{j_1}a,x)-l(b_{j_2}a,x)\leq Ca^2
\end{equation*}
for some uniform constant $C$, according to Lemma \ref{lemma}.

But then 
\begin{equation*}
0\leq \cos(\theta(x,v)) \leq \langle (v_{j_1}-v_{j_2}),n(x) \rangle \leq   Ca
\end{equation*}
where $\theta (x,v)$ is the angle from $n(x)$ to $v_{j_1}-v_{j_2}$. Thus, $\theta (x,v)=\theta (x,0)+v$ must lie in $\cos^{-1}([0,Ca])$. But
\begin{equation*}
\Ha^1(v\in [0,2\pi]\mid \theta (x,v)\in \cos^{-1}([0,Ca]))=\Ha^1(\cos^{-1}([0,Ca])\cap[0,2\pi])\leq C'a
\end{equation*}
 and there are only $6$ possible combinations of $j_1$ and $j_2$, so
\begin{equation*}
a^{-1}\Ha^1(v\in [0,2\pi]\mid \exists i: v_i\neq v_i') \leq a^{-1}6\Ha^1(\cos^{-1}([0,Ca])\cap[0,2\pi]) \leq 6C'.
\end{equation*}

Suppose $t_i<t_i'=t_j$. If $j<i$, the last claim of the lemma follows from Lemma \ref{lemma} and \eqref{prime} as $a(s_{j_2}-s_{j_1})$ is negative. If $i<j$, there must be a $k<i$ with $t_j<t_k$. Then
\begin{equation*}
|t_i-t_i'|\leq |t_i-t_k|+|t_k-t_j|\leq 4\su \{|l(ba,x)| \mid (b,x)\in [-\sqrt{2},\sqrt{2}]\times  \partial X\}
\end{equation*}
by a double application of \eqref{prime}. The case $t_i>t_i'$ can be treated in a  similar way.
\end{proof} 

%
%In order to prove Theorem \ref{EC}, we need to describe the asymptotic behavior of $EN_j$. This is computed by integrating over all rotations in \eqref{intform} and letting $a$ tend to 0.
%The two terms on the right hand side of \eqref{intform} are treated separately in the two next lemmas.

We are now ready to prove the two main Lemmas.

\begin{lem}\label{l1}With $f_l$ as in \eqref{fl},
\begin{equation*}
\lim_{a\to 0}{a^{-2}}\sum_{l: \xi_l \in \eta_j}\frac{1}{2\pi} \int_0^{2\pi}\int_{\partial X}\int_{-\eps}^{\eps}tf_l(x+tn,v) k(x)dt \mathcal{H}^{1}(dx)dv=\begin{cases}V_0(X),&j=2,\\
0,&j=3,4,\\
-V_0(X),&j=5.\\
\end{cases}
\end{equation*}
%exists and equals $1$ for $j=2$, $0$ for $j=3$ and $-1$ for $j=5$.
\end{lem}

\begin{proof}
%By Fubini's theorem 
%\begin{align*}
%%\int_{0}^{2\pi}\int_{\partial X} I_l^1 d\Ha^{1}dv =  
%\int_0^{2\pi}&\int_{\partial X}\int_{-\eps}^{\eps} tf_l(x+tn,v)k(x)dt\Ha^{1}(dx)dv=
%\int_{\partial X}\int_0^{2\pi}\int_{-\eps}^{\eps} tf_l(x+tn,v)k(x)dt\Ha^{1}(dx)dv.
%%
%%\\=&\int_{\partial X}\int_0^{2\pi}\int_{-\eps}^{\eps}t1_{\{x+tn+aR_{v}(B_l)\subseteq X\}}1_{\{x+tn+aR_{v}(W_l)\subseteq \R^2\backslash X\}}dtdvk(x)\Ha^{1}(dx)\\
%%=&\int_{\partial X}\int_0^{2\pi}\int_{-\eps}^{\eps}t1_{\{x+tn+aR_{v-\theta(n)}(B_l)\subseteq X\}}1_{\{x+tn+aR_{v-\theta(n)}(W_l)\subseteq \R^2\backslash X\} }dtdvk(x)\Ha^{1}(dx).
%%\\=\int_{\partial X}\int_{-\eps}^{\eps} \int_0^{2\pi}&t1_{\{R_{-n}(x)+t(1,0)+aR_{v}(B_l)\subseteq R_{-n}(X)\}}\\ &\times 1_{\{R_{-n}(x)+t(1,0)+aR_{v}(W_l)\subseteq \R^2\backslash R_{-n}(X)\}}dvdt\Ha^{1}(dx).
%\end{align*}
%In particular, we may assume $n(x)=(1,0)$ for simplicity. 
For $x\in \partial X$ fixed, let
\begin{equation*}
I_j(x,v)= \sum_{l:\xi_l \in \eta_j} \int_{-\eps}^{\eps}  tf_l(x+tn,v) dt.
\end{equation*}  

%There is a number of special cases to consider. First suppose $u \in (-\frac{\pi}{2},-\frac{\pi}{4})$.
%With the notation $C_0=\{x_0,x_1,x_2,x_3\}$ introduced in Section~\ref{notation} and $l$ defined as in Lemma \ref{lemma}, we observe that for $w=u+\frac{\pi}{2}$, 
%\begin{align*}
%&x+tn+ aR_u(x_0)\in X & &\textrm{ for } t\leq t_0:= 0,\\
%&x+tn+ aR_u(x_1)\in X & &\textrm{ for } t\leq   t_1:=- (a\sin w + l(a\cos w,x)),\\
%&x+tn+ aR_u(x_2)\in X & &\textrm{ for } t\leq t_2:= -(a\cos w +l(-a\sin w ,x)),\\
%&x+tn+ aR_u(x_3)\in X & &\textrm{ for } t\leq  t_3:= -(a\cos w + a\sin w +l(a(\cos w-\sin w),x)) .
%\end{align*}
%Let $t_0'\leq t_1'\leq t_2' \leq t_3'$ be the ordering of the numbers $t_0,\dots , t_3$(this may depend on $a$). 
%Then $x + tn +aR_u(C_0)$ is a configuration of the following type
For $\sqrt{2}a<\eps$, configurations of type $\eta_4$ can never occur, so $(x + tn +aR_v(C_0))\cap X$ corresponds to a configuration of  type
\begin{equation*}
\eta_1 \textrm{ for } t < t_3',\,
\eta_2 \textrm{ for } t\in (t_2',t_3'],\,
\eta_3 \textrm{ for } t\in (t_1',t_2'],\,
\eta_5 \textrm{ for } t\in (t_0',t_1'],\,
\textrm{ and }\eta_6 \textrm{ for } t\leq  t_0',
\end{equation*}
according to \eqref{smut}.

%$x+tn+ R_v(x_0)\in X$ for $t\leq 0$, $x+tn+ R_v(x_1)\in X$ for $t\leq - (a\sin(v+\frac{\pi}{2}) + l(s(\cos(v+\frac{\pi}{2})a),t_0))$, $x+tn+ R_v(x_2)\in X$ for $t\leq -(a\cos(v+\frac{\pi}{2}) +l(s(-\sin(v+\frac{\pi}{2})a),t_0))$ and finally $x+tn+ R_v(x_3)\in X$ for $t\leq -(a\cos(v+\frac{\pi}{2}) + a\sin(v+\frac{\pi}{2}) +l(s((\cos(v+\frac{\pi}{2})-\sin(v+\frac{\pi}{2}))a),t_0))$.

As an example, consider the configuration type $\eta_5$. Then we get
\begin{equation*}
I_5 = \int_{t_0'}^{t_1'}t dt  = \tfrac{1}{2}(t_1^{\prime2}-t_0^{\prime 2}).
\end{equation*}
By Fubini's theorem we must compute 
\begin{equation*}
\lim_{a \to 0} {a^{-2}} \int_{\partial X}\int_{0}^{2\pi} I_5 dvk d\mathcal{H}^{1}  = \lim_{a \to 0} {a^{-2}} \int_{\partial X}\int_{0}^{2\pi} \frac{1}{2}(t_1^{\prime2}-t_0^{\prime 2}) dv k d\mathcal{H}^{1}.
\end{equation*}
%By Lemma \ref{lemma} and \ref{order}, $a^{-2}(t_i^2-t_i^{\prime 2})=a^{-2}(t_i-t_i')(t_i+t_i')$ is uniformly bounded. Thus by Lebesgue's theorem we may replace $t_i'$ by $t_i$ in the integral. 

By Lemma \ref{order},  $\lim_{a\to 0}\Ha^1(v\in[0,2\pi ] \mid t_i\neq t_i')=0$ uniformly.
Moreover, it follows from Lemma \ref{lemma} that 
\begin{equation*}
{a^{-2}}t_i^{ 2} = s_i^2 - 2s_ia^{-1}l(b_ia,x)+ a^{-2}l(b_ia,x)^2 
\end{equation*}
is uniformly bounded.  Hence we may replace $t_i^{\prime 2}$ by $t_i^{2}$ in the integral by the Lebesgue theorem of dominated convergence. This also applies to give
\begin{align*}
%\lim_{a \to 0} \frac{1}{a^2} \frac{1}{2\pi}\int_{0}^{\frac{\pi}{2}}\int\int_0^{\epsilon}tf(x+tn) dt k(x) d\mathcal{H}^{d-1}dv  = \frac{1}{2\pi}\int_{0}^{\frac{\pi}{2}} \lim_{a \to 0} \frac{1}{a^2} \frac{1}{2}{t_1}^2-{t_2}^2 \mathcal{H}^{d-1}dv
\lim_{a \to 0} {a^{-2}}\int_{\partial X}\int_{0}^{2\pi}I_5 dvk d\mathcal{H}^{1}  =& \int_{\partial X}\int_{0}^{2\pi} \lim_{a \to 0} {a^{-2}}\cdot \frac{1}{2}(t_1^{2}-t_0^{2}) dv k d\mathcal{H}^{1}%\label{pi4int}
\\=& \int_{\partial X}\int_{0}^{2\pi} \frac{1}{2}(s_1^2- s_0^2 )dvkd\mathcal{H}^{1}. %\nonumber
%\\=&2\pi V_0(X) \frac{1}{2}\int_{0}^{\frac{\pi}{4}} (s_1^2- s_0^2)dw.\nonumber
\end{align*}
The last step used Lemma \ref{lemma}.
% and the fact that $t_0'=t_0$ and $t_1'=t_1$ for $a$ sufficiently small. The latter follows by an application to $l$ of the mean value theorem.

%This is only the contribution for $u \in (-\frac{\pi}{2},-\frac{\pi}{4})$. 
%Now we must determine the $s_i$ for various values of $u$.
%For $u\in (-\frac{5\pi}{4},-\frac{\pi}{4})$ we get:
%\begin{align*}
%& (s_1^{2}-s_0^{2})=\begin{cases}
% \cos^2 w - (\cos w + \sin w)^2 & \textrm{ for }w = u+\frac{\pi}{2}\in (0,\frac{\pi}{4}),\\
%(\cos w-\sin w)^2-\cos^2w&\textrm{ for } w=-\frac{\pi}{2}-u\in (0,\frac{\pi}{4}),\\
%-\sin^2w&\textrm{ for } w=u+\pi \in (0,\frac{\pi}{4}),\\
%\sin^2w&\textrm{ for } w=-u-\pi \in  (0,\frac{\pi}{4}).
%\end{cases}\\
%&(s_2^{2}-s_1^{2})=\begin{cases}
% \sin^2w-\cos^2w &\textrm{ for } w = u+\frac{\pi}{2}\in (0,\frac{\pi}{4}),\\
%-(\cos w-\sin w)^2& \textrm{ for }w=-\frac{\pi}{2}-u\in (0,\frac{\pi}{4}),\\
%(\cos w-\sin w)^2&\textrm{ for } w=u+\pi \in (0,\frac{\pi}{4}),\\
%\cos^2w-\sin^2w&\textrm{ for } w=-u-\pi \in  (0,\frac{\pi}{4}).
%\end{cases}\\
%&(s_3^{2}-s_2^{2})=\begin{cases}
%-\sin^2w &\textrm{ for } w = u+\frac{\pi}{2}\in (0,\frac{\pi}{4}),\\
%\sin^2w&\textrm{ for } w=-\frac{\pi}{2}-u\in (0,\frac{\pi}{4}),\\
%\cos^2w-(\cos w-\sin w)^2& \textrm{ for }w=u+\pi \in (0,\frac{\pi}{4}),\\
%(\cos w+\sin w)^2-\cos^2 w& \textrm{ for }w=-u-\pi \in  (0,\frac{\pi}{4}).
%\end{cases}
%\end{align*}
% We get the same contribution for $u$ and $-\frac{\pi}{2}-u$ by symmetry, so this takes care of the case $u\in (-\frac{\pi}{4},\frac{3\pi}{4})$.

Substituting $u=R_{-v}n$ and inserting the values of $s_i(u)$, a direct computation shows:
\begin{align*}
\lim_{a \to 0}{a^{-2}}\int_{\partial X}\int_{0}^{2\pi} I_5(x,v) dv k(x)\mathcal{H}^{1}(dx)  ={}& \int_{\partial X}
\int_{S^1}\frac{1}{2}(s_1^2(u)-s_0^2(u)) du k d\mathcal{H}^{1} \\
%\\={}&
%=&2\pi V_0(X)\cdot 2  \int_0^{\frac{\pi}{4}}\frac{1}{2}\bigg((\cos^2 w - (\cos w + \sin w)^2)\\
% &+ ((\cos w - \sin w)^2 - \cos^2 w  ) + (-\sin^2w) + \sin^2w \bigg)dw \\
  %2\pi V_0(X)\int_0^{\frac{\pi}{4}} (-4\cos w\sin w )dw \\
  ={}& -2\pi V_0(X).
\end{align*}

The remaining configuration types $\eta_2$ and $\eta_3$ are treated similarly.
%Similarly for the remaining configuration types:
%\begin{align*}
%\int_{0}^{2\pi}\lim_{a \to 0}\frac{1}{a^2} I_3 du  = \int_{0}^{2\pi} \frac{1}{2}(s_2^2-s_1^2)du &=2\int_0^{\frac{\pi}{4}} 0dw =0,\\
%\int_{0}^{2\pi}\lim_{a \to 0}\frac{1}{a^2} I_2 du = \int_{0}^{2\pi}\frac{1}{2}(s_3^2-s_2^2) du &= 2\int_0^{\frac{\pi}{4}} \frac{1}{2}\cdot 4 \cos w\sin w dw = 1,
%\end{align*}
%and the claim follows.
\end{proof}

\begin{lem}\label{l2}
For $w_j^{(i)}\in \R$ and $c_3^{(i)}$ as in \eqref{constants}, the limit
\begin{equation*}
\lim_{a\to 0} {a^{-2}}\cdot\frac{1}{2\pi} \bigg(\sum_{j=2}^5 w_j^{(i)}   \int_{\partial X}\int_{0}^{2\pi} \int_{-\eps}^\eps \sum_{l:\xi_l \in \eta_j}f_l(x+tn,v)dt dv \mathcal{H}^{1}(dx)
  - 2{a}c_3^{(i)}V_1(X)  \bigg) 
\end{equation*}
exists and equals
\begin{equation*}
(w_2^{(i)}-w_5^{(i)})V_0(X).
\end{equation*}
\end{lem}

\begin{proof}
Let $x\in \partial X$ be given and define
\begin{equation*}
I_j(x,v)= \sum_{l:\xi_l \in \eta_j} \int_{-\eps}^{\eps}f_l(x+tn,v) dt.
\end{equation*}
%%As in the proof of Lemma \ref{l1}, there are various situations to consider.
%%Suppose first, as an example, that $u = v-n(x) \in (-\frac{\pi}{2},-\frac{\pi}{4})$ and $j=5$.  
By the same reasoning as in the proof of Lemma \ref{l1},
\begin{equation*}
I_2 =t_3'-t_2'\text{, }I_3 =t_2'-t_1'\text{, and }I_5 =t_1'-t_0'.
\end{equation*}
As an example, consider $\eta_5$. We shall compute
\begin{align}
\begin{split}\label{i5}
\lim_{a\to 0}&a^{-2}\int_{\partial X}\int_{0}^{2\pi}( I_5 +a(s_1-s_0) )dvd\mathcal{H}^{1}\\ &= \lim_{a\to 0}
\int_{\partial X}\int_{0}^{2\pi}({a^{-2}}(t_1'-t_0')+ {a^{-1}}(s_1-s_0) )dvd\mathcal{H}^{1}.
\end{split}
\end{align}
% The triangle inequality yields
%\begin{align*}
%|{a^{-2}}(t_{i+1}'-t_i') + {a^{-1}}(s_{i+1}-s_i)|
%\leq |{a^{-2}}(t_{i+1}-t_i) +& {a^{-1}}(s_{i+1}-s_i)|\\&+{a^{-2}}|t_{i+1}'-t_{i+1}|+{a^{-2}}|t_i-t_i'|.
%\end{align*}
%Each term on the right hand side is bounded by Lemma \ref{lemma} and \ref{order}.
Since $a^{-2}|t_i-t_i'|\leq M'$ and $\Ha^{1}(t_i\neq t_i')<Ma$ by Lemma \ref{order} for some uniform constants $M$ and $M'$, we may replace $t_i$ by $t_i'$ in \eqref{i5}.

%Considering the possible special cases, one realizes that
%\begin{equation*}
%|t_i-t_i'|\leq 2\su \{| l(r,x)| \mid x\in \partial X ,r \in [-\sqrt{2}a,\sqrt{2}a]\}
%\end{equation*}
%for all $i$.
%Thus $\frac{1}{a^2}|t_i-t_i'|$ is uniformly bounded for $(x,u,a)\in \partial X \times (-\frac{\pi}{2},-\frac{\pi}{4})\times (0,\frac{\eps}{\sqrt{2}}]$ by Lemma \ref{lemma}. 
%Hence by another application of that lemma, so is
%\begin{equation*}
%|\frac{1}{a^2}(t_1'-t_0') - \frac{1}{a}\sin u|.
%\end{equation*}
By another application of Lemma \ref{lemma}, 
\begin{equation*}
a^{-2}t_i+a^{-1}s_i=a^{-2}l(b_ia,x)
\end{equation*}
 is uniformly bounded. This allows us to apply Lebesgue's theorem to \eqref{i5}. In the case of  $\eta_5$, this yields
\begin{align*}%\label{limin}
\lim_{a\to 0}&\int_{\partial X}\int_{0}^{2\pi}( {a^{-2}}I_5 +{a^{-1}}(s_1-s_0) )dvd\mathcal{H}^{1}\\  &=
\int_{\partial X}\int_{0}^{2\pi}\lim_{a\to 0}({a^{-2}}(t_1'-t_0')+ {a^{-1}}(s_1-s_0) )dvd\mathcal{H}^{1}\\ &=
\int_{\partial X}\int_{0}^{2\pi}\lim_{a\to 0}{a^{-2}}(l(ab_1,x)-l(ab_0,x))dv\mathcal{H}^{1}(dx)\\  &=
\int_{\partial X}\int_{0}^{2\pi}\frac{-k}{2}(b_1^2-b_0^2)dvd\mathcal{H}^{1}
\end{align*}
where the last step also follows from Lemma \ref{lemma}. 
%The same argument works for the remaining angles $u\in(-\frac{5\pi}{4},-\frac{\pi}{4})$ and configuration types, using
%\begin{align*}
%-\frac{2}{k(x)}\lim_{a \to 0}( \frac{1}{a^2} (t_1-t_0)-\frac{1}{a}\sin w )=\begin{cases}
%\sin^2 w - (\cos w - \sin w)^2  , & w = u+\frac{\pi}{2}\in (0,\frac{\pi}{4})\\
%(\cos w +\sin w)^2-\sin^2w,& w=-\frac{\pi}{2}-u\in (0,\frac{\pi}{4})\\
%-\cos^2w,& w=u+\pi \in (0,\frac{\pi}{4})\\
%\cos^2w,& w=-u-\pi \in  (0,\frac{\pi}{4})
%\end{cases}\\
%-\frac{2}{k(x)}\lim_{a \to 0}( \frac{1}{a^2} (t_2-t_1)-\frac{1}{a}(\cos w-\sin w) )=\begin{cases}
% \cos^2w - \sin^2 w  , & w = u+\frac{\pi}{2}\in (0,\frac{\pi}{4})\\
%-(\cos w+\sin w)^2,& w=-\frac{\pi}{2}-u\in (0,\frac{\pi}{4})\\
%(\cos w+\sin w)^2,& w=u+\pi \in (0,\frac{\pi}{4})\\
%\sin^2w-\cos^2w,& w=-u-\pi \in  (0,\frac{\pi}{4})
%\end{cases}\\
%-\frac{2}{k(x)}\lim_{a \to 0}( \frac{1}{a^2} (t_1-t_0)-\frac{1}{a}\sin w )=\begin{cases}
% -\cos^2w, & w = u+\frac{\pi}{2}\in (0,\frac{\pi}{4})\\
%\cos^2w,& w=-\frac{\pi}{2}-u\in (0,\frac{\pi}{4})\\
%\sin^2w-(\cos w+\sin w)^2,& w=u+\pi \in (0,\frac{\pi}{4})\\
%(\cos w-\sin w)^2-\sin^2 w,& w=-u-\pi \in  (0,\frac{\pi}{4})
%\end{cases}\\
%\end{align*}
%For $u\in (-\frac{\pi}{4},\frac{3\pi}{4})$ we get the same contribution as we do for $-u-\frac{\pi}{2}$.

Doing the same for the remaining configurations, a computation shows that
\begin{align} \label{exist}
 -\int_{\partial X}&\int_{0}^{2\pi}\frac{k}{2}(w_2^{(i)}(b_3^2-b_2^2)+w_3^{(i)}(b_2^2-b_1^2)+w_5^{(i)}(b_1^2-b_0^2))dvd\mathcal{H}^{1}\\
=&\lim_{a\to 0} {a^{-2}}  \int_{\partial X}\int_{0}^{2\pi}\bigg(\sum_{j=2}^5 w_j^{(i)} I_j \nonumber \\  \nonumber & \qquad - {a}(w_2^{(i)} (s_2-s_3) +w_3^{(i)}(s_1-s_2)+w_5^{(i)} (s_0-s_1))\bigg)dvd\mathcal{H}^{1}\\\nonumber
%=&\lim_{a\to 0} {a^{-2}}\bigg( \sum_{j=2}^5 w_j^{(i)}  \int_{\partial X}\int_{0}^{2\pi} I_jdvd\mathcal{H}^{1}\\ \nonumber & \qquad - {a}\mathcal{H}^{1}(\partial X)8\bigg( \int_{0}^{\frac{\pi}{4}}(w_2^{(i)}\sin w +w_3^{(i)} (\cos w- \sin w)+w_5^{(i)}\sin w) dw\bigg)\bigg)\\ \nonumber 
%=& \lim_{a\to 0} {a^{-2}}\bigg( \sum_{j=2}^5 w_j^{(i)}  \int_{\partial X}\int_{0}^{2\pi } I_j dvd\mathcal{H}^{1}\\ \nonumber & \qquad -  2{a}V_1(X)\big( (8-4\sqrt{2})w_2^{(i)} + (8\sqrt{2}-8)w_3^{(i)} + (8-4\sqrt{2})w_5^{(i)} \big) \bigg).
=& \lim_{a\to 0} {a^{-2}}\bigg( \sum_{j=2}^5 w_j^{(i)}  \int_{\partial X}\int_{0}^{2\pi } I_j dvd\mathcal{H}^{1} -  2{a}c_3^{(i)}V_1(X) \bigg).
\end{align}

%On the other hand, inserting the $b_i$ from Table \ref{tabel}, a direct computation shows:
%\begin{align*}
%&\int_{0}^{2\pi}(b_1^2-b_0^2)dv = 2\int_{0}^{\frac{\pi}{4}} 4\sin w \cos w dw = 2,\\
%&\int_{0}^{2\pi}(b_2^2-b_1^2)dv = 2\int_{0}^{\frac{\pi}{4}} 0 dw=0,\\
%&\int_{0}^{2\pi}(b_3^2-b_2^2)dv = 2\int_{0}^{\frac{\pi}{4}} (-4\sin w \cos w )dw = -2.
%\end{align*}
On the other hand, another computation shows that \eqref{exist} equals
\begin{equation*}
%\lim_{a\to 0} \sum_j w_j  \int_{\partial X}\int_{0}^{2\pi}\sum_{l:\xi_l \in \eta_j} I_l^2dv \mathcal{H}^{d-1}(dx)&= 
-\int_{\partial X}\frac{k(x)}{2} (-2w_2^{(i)}+2w_5^{(i)})\mathcal{H}^{1}(dx)= 2 \pi V_0(X)(w_2^{(i)}-w_5^{(i)}),
\end{equation*}
from which the claim follows.
\end{proof}

\begin{proof}[Proof of Theorem \ref{EC}]
From Lemma \ref{l1} and \ref{l2}, it follows that the limit
\begin{align}\label{2int}
\lim_{a\to 0}&\big( a^{-i}E\hat{V}_i(X)-a^{-1}\tfrac{1}{\pi}c_3^{(i)}V_1(X)\big)\\ 
%\nonumber &=
%\lim_{a\to 0}\bigg( \sum_{j=2}^{5}w_j^{(i)} EN_j-a^{-1}\frac{1}{\pi}c_3^{(i)}V_1(X)\bigg)\\
 \nonumber &= \lim_{a\to 0} {a^{-2}}\bigg(\sum_{j=2}^5w_j^{(i)} \sum_{l:\xi_l\in\eta_j}\frac{1}{2\pi}\int_0^{2\pi }\bigg(  \int_{\partial X}\int_{-\eps}^{\eps}tf_l(x+tn,v) k(x)dt \mathcal{H}^{1}(dx)\\ 
 &\qquad  +\int_{\partial X}\int_{-\eps}^{\eps}f_l(x+tn,v) dt\mathcal{H}^{1}(dx)\bigg)dv-a\tfrac{1}{\pi}c_3^{(i)}V_1(X)\bigg)\nonumber
\end{align}
exists and equals $c_4^{(i)}V_0(X)$. 
%Hence $\lim_{a\to 0} E\hat{V}_0(X)$ exists if and only if $c_3^{(0)}=0$, and in this case the limit equals $c_4^{(0)}V_0(X)$.

In the limit, the condition \eqref{item2} is
\begin{align*}
\lim_{a\to 0}E\hat{V}_0(X)=\lim_{a\to 0} (w_2^{(0)}EN_2(X)+ w_3^{(0)}EN_3(X) +w_5^{(0)}EN_5(X)) &= V_0(X),\\
\lim_{a\to 0}E\hat{V}_0(\R^2\backslash X)=\lim_{a\to 0} (w_2^{(0)}EN_5(X)+ w_3^{(0)}EN_3(X) +w_5^{(0)}EN_2(X)) &= -V_0( X).
\end{align*}
This is equivalent to
\begin{align*}
\lim_{a\to 0} (w_2^{(0)}EN_2+ w_3^{(0)}EN_3 +w_5^{(0)}EN_5) &= V_0(X),\\
\lim_{a\to 0} (w_2^{(0)}-w_5^{(0)})(EN_2-EN_5)& = 2V_0(X).
\end{align*}
From \eqref{2int} with $w_2^{(0)}=1$, $w_3^{(0)}=w_4^{(0)}=0$, and $w_5^{(0)} = -1$, it follows that 
\begin{equation*}
\lim_{a\to 0}(EN_2-EN_5) = 4V_0(X).
\end{equation*}
Thus Equation \eqref{w22} ensures that \eqref{item2} holds asymptotically.
%The statement about $\hat{V}_1(X)$ follows from \eqref{2int} in a similar way.
\end{proof}

When $\partial X$ is actually a $C^3$ manifold, we can get slightly better asymptotic results:
\begin{thm}
Let $ X \subseteq \R^2$ be a $C^3 $ full-dimensional submanifold. Assume that the weights defining $\hat{V}_1(X)$  satisfy Equations \eqref{w11} and \eqref{w12} and the weights defining  $\hat{V}_0(X)$ satisfy Equations \eqref{w21} and \eqref{w22}. Then $E\hat{V}_1(X)$ and $E\hat{V}_0(X)$ converge as $O(a^2)$ and $O(a)$, respectively.
\end{thm}

\begin{proof}
It is enough to check that ${a^{-i-1}}(E\hat{V}_i(X)-\lim_{a\to 0}E\hat{V}_i(X))$ is bounded.
Going through the proofs of Lemma \ref{l1} and \ref{l2}, we see that it is enough to show that 
\begin{equation}\label{int1}
{a^{-3}}(t_{i+1}^{\prime 2}-t_{i}^{\prime 2}) -{a^{-1}}(s_{i+1}^{2}-s_{i}^{2}) 
\end{equation}
and
\begin{equation}\label{int2}
{a^{-1}}\int_0^{2\pi} \bigg({a^{-2}}(t_{i+1}^{\prime }-t_{i}^{\prime })-{a^{-1}}(s_i-s_{i+1}) +\frac{k}{2}(b_{i+1}^2-b_i^2) \bigg)dv
\end{equation}
are uniformly bounded.

The triangle inequality yields
\begin{equation*}
|{a^{-3}}t_{i}^{\prime 2} - {a^{-1}}s_{i}^{2}| 
\leq 
|{a^{-3}}t_i^{ 2}- {a^{-1}}s_{i}^{2}|
%+|{a^{-3}}t_{i+1}^{ 2}- {a^{-1}}s_{i+1}^{2}|\\ &
+ {a^{-3}}|t_i^{\prime 2}-t_{i}^{ 2}| .
%+{a^{-3}}|t_{i+1}^{ 2}-t_{i+1}^{\prime 2}|
\end{equation*}
The terms
\begin{equation*}
|{a^{-3}}t_i^{2}-{a^{-1}}s_i^2| =|-2s_ia^{-2}{l(b_ia,x)} + a^{-3}{l(b_ia,x)^2}|
\end{equation*}
are uniformly bounded by Lemma \ref{lemma}. Furthermore, 
\begin{equation*}
\frac{|t_i^{\prime 2}-t_{i}^{ 2}|}{a^3}=\frac{|t_i^{\prime}+t_{i}|}{a}\frac{|t_i^{\prime}-t_{i}|}{a^2}
\end{equation*}
is bounded by Lemma \ref{order}. This takes care of \eqref{int1}.

Similarly,
\begin{equation*}
\big|{a^{-3}}t_{i}^{\prime }+{a}^{-2}s_i +a^{-1}\tfrac{k}{2}b_i^2 \big| \leq \big|{a^{-3}}t_{i}+a^{-2}s_i +a^{-1}\tfrac{k}{2}b_i^2 \big|+ {a^{-3}}|t_i-t_i'|.
\end{equation*}
%\begin{align*}
%\big|{a^{-3}}(t_{i+1}^{\prime }&-t_{i}^{\prime })-{a}^{-3}(s_i-s_{i+1}) +a^{-1}\tfrac{k}{2}(b_{i+1}^2-b_i^2) \big|\\ \leq &\big|{a^{-3}}(t_{i+1}-t_{i})-{a^{-2}}(s_i-s_{i+1}) +a^{-1}\tfrac{k}{2}(b_{i+1}^2-b_i^2) \big|\\ &\qquad+ {a^{-3}}|t_i-t_i'| + {a^{-3}}|t_{i+1}-t_{i+1}'|.
%\end{align*}
Again by Lemma \ref{order}, ${a^{-2}}|t_i-t_i'|$ is uniformly bounded by some $C$ and hence
\begin{equation*}
\int_0^{2\pi}{a^{-3}}|t_{i}-t_{i}'|dv \leq \int_0^{2\pi}{a^{-1}}C 1_{\{t_i\neq t_i'\} }dv
\end{equation*}
is also uniformly bounded by Lemma \ref{order}.  
Finally,
\begin{equation*}
{a^{-3}}t_{i}+{a^{-2}}s_i + a^{-1}\tfrac{k}{2}b_i^2 ={a^{-3}}l(b_{i}a,x) + a^{-1}\tfrac{k}{2}b_i^2.
\end{equation*}
But by a refinement of Lemma \ref{lemma}, $r\mapsto l(r,x)$ is $C^3$ when $\partial X$ is a $C^3$ manifold and 
\begin{equation*}
\frac{l(br,x)}{r^3}+\frac{b^2k(x)}{2r}
\end{equation*}
is bounded for $(b,r,x)\in [-\sqrt{2},\sqrt{2}]\times [-\delta,\delta]\backslash \{0\}\times \partial X$. This takes care of \eqref{int2}.
\end{proof}

\section{Classical choices of weights}\label{classical}
Recall that for a stationary isotropic Boolean model $\Xi$ with grain distribution satisfying~\eqref{curvbound} a.\ s., we found in Theorem \ref{w1} that
\begin{equation*}
\lim_{a\to 0} E\hat{V}_1(\Xi) = \tfrac{1}{\pi}c_3^{(1)}\altoverline{V}_1(\Xi).
\end{equation*}
%In particular, $\hat{V}_1(\Xi)$ is asymptotically unbiased if and only 
If $c_3^{(1)}=\pi$, the bias for small values of $a$ is approximately
\begin{align*}
E\hat{V}_1(\Xi)-\altoverline{V}_1(\Xi)\approx a\left( c_4^{(1)}\gamma + c_5^{(1)}\left(\tfrac{\gamma}{\pi}EV_1(C)\right)^2e^{-\gamma EV_2(C)}\right)
\end{align*}
with $c_m^{(1)}$ as in \eqref{constants}.

In the literature, various local algorithms are used for estimating the boundary length of a planar set. With the formulas above we can compute their asymptotic bias and thus compare their accuracy. 

Ohser and M\"{u}cklich, \cite{OM}, describe an estimator for $\altoverline{V}_1(\Xi)$ based on a discretized version the Cauchy projection formula. In the rotation invariant setting, the estimator corresponds to \eqref{Nest} with weights:
\begin{equation*}
w^{(1)} = \big(0,\tfrac{\pi}{16}\big(1+\tfrac{\sqrt{2}}{2}\big),\tfrac{\pi}{16}(1+\sqrt{2}) ,\tfrac{\pi}{8},\tfrac{\pi}{16}\big(1+\tfrac{\sqrt{2}}{2}\big),0 \big).
\end{equation*}
Inserting these weights in the equations shows that this estimator satisfies \eqref{w11} and is thus asymptotically unbiased. The weights also satisfy \eqref{w12} but not \eqref{w13}. For small values of $a$, the error is approximately
\begin{equation*}
- a\tfrac{1+\sqrt{2}}{2} \tfrac{\gamma^2}{\pi} EV_1(C)^2 e^{-\gamma EV_2(C)}\approx - 1,207 a\tfrac{\gamma^2}{\pi} EV_1(C)^2 e^{-\gamma EV_2(C)}.
\end{equation*}

%
%A very simple algorithm \fixme{reference}for estimating the boundary length based only on $2\times 1$ configurations is simply to count the number $N$ of of occurances of the configuration $\begin{pmatrix}\bullet & \circ\end{pmatrix}$, or equivalently the number of 
%$\begin{pmatrix} \bullet & \circ \\ \cdot & \cdot\end{pmatrix}$. It is well known that $\lim_{a\to 0}a\frac{\pi}{2}N=\overline{V}(\Xi)$. Averaging over rotations of this, it corresponds to the weights

One of the oldest algorithms for estimating the boundary length is suggested by Bieri in \cite{bieri}. The idea is to approximate the underlying object by a union of squares of side length $a$ centered at the foreground pixels and use the boundary length of the approximation as estimate. This corresponds to a local estimator with  weights
\begin{equation*}
w^{(1)} = \big(0,\tfrac{1}{2},\tfrac{1}{2} ,1,\tfrac{1}{2},0\big).
\end{equation*}
However, it is well-known that for a compact object $X$ this is the boundary length of the smallest box containing $X$ and hence  is a very coarse estimate. The asymptotic mean is
$\tfrac{4}{\pi}\altoverline{V}_1(X)$. 
Of course, one can correct for the factor $\frac{4}{\pi}$ and consider the weights 
\begin{equation}\label{cauchy}
w^{(1)} = \big(0,\tfrac{\pi}{8},\tfrac{\pi}{8},\tfrac{\pi}{4},\tfrac{\pi}{8},0\big)
\end{equation}
instead. These weights can be justified by the Cauchy formula  in \cite{OM} using $\theta_1=\frac{\pi}{2}$. It is also the unique unbiased estimator where all weights are equal, except that configurations of type $\eta_4$ are counted with double weight. These weights satisfy Equations \eqref{w11} and \eqref{w12} but not \eqref{w13}. The bias for small $a$ is approximately
\begin{equation*}
-a\tfrac{\gamma^2}{\pi} EV_1(C)^2e^{-\gamma EV_2(C)}.
\end{equation*}

The approach of Dorst and Smeulders in \cite{dorst} is to reconstruct the underlying set by an 8-adjacency system and compute the length of the boundary of the reconstructed set, letting vertical and horizontal segments contribute with one weight and diagonal segments with another weight. The resulting estimators are of the forms
\begin{align}\label{graph}
\begin{split}
w^{(1)} &= \big(0,0,\tfrac{\theta}{2}, {\sqrt{2}\theta},\tfrac{\sqrt{2}\theta}{2},0\big),\\
w^{(1)} &= (0,0,\alpha,2\beta,\beta,0).
\end{split}
\end{align} 
These algorithms are only tested on straight lines in \cite{dorst} and therefore it was not necessary to assign a value $w_4^{(1)}$. The weights chosen here are such that a diagonal segment coming from a configuration of type $\eta_4$ is counted double.  

The authors list some of the constants frequently used in the literature.
The case $\theta =1$ goes back to Freeman in \cite{freeman}. This yields a biased estimator. But even if the constants are chosen such that the estimator is asymptotically unbiased, all weights of this form have the disadvantage of not satisfying Equation \eqref{w12}, which is the most desirable of the two equations \eqref{w12} and \eqref{w13}, as it also appears in the design based setting. 

The boundary is also sometimes approximated using a 4- or 6-adjacency graph. However, the same problem with Equation \eqref{w12} arises. 

Another classical approach is the marching squares algorithm. This is based on a reconstruction of both foreground and background. The boundary is then approximated by a digital curve lying between these, see e.g.\ \cite{digital}, Figure 4.29. The corresponding weights are
\begin{equation*}
w^{(1)} = \big(0,\tfrac{\sqrt{2}}{4},\tfrac{1}{2}, \tfrac{\sqrt{2}}{2},\tfrac{\sqrt{2}}{4},0\big).
\end{equation*}
This estimator is not asymptotically unbiased either. In fact, the asymptotic mean is
\begin{equation*}
(2\sqrt{2}-2)\tfrac{4}{\pi}\altoverline{V}_1(\Xi) \approx 1,0548\altoverline{V}_1(\Xi).
\end{equation*} 
Correcting for this factor, we obtain an asymptotically unbiased estimator satisfying Equation \eqref{w21} with approximate bias for small values of $a$ 
\begin{equation*}
a\tfrac{\sqrt{2}-6}{4} \tfrac{\gamma^2}{\pi} EV_1(C)^2e^{-\gamma EV_2(C)}\approx -1,146a \tfrac{\gamma^2}{\pi} EV_1(C)^2e^{-\gamma EV_2(C)}.
\end{equation*}

Similarly, one can compare the classical estimators for $V_0$.
Ohser and M\"{u}cklich suggest an estimator  in  \cite{OM} based on the approximation of $\Xi$ by a $6$-neighborhood graph. This results in weights
\begin{equation}\label{OMEuler}
w^{(0)} = \big(0,\tfrac{1}{4},0,0,-\tfrac{1}{4},0\big).
\end{equation}
These satisfy \eqref{w21} and \eqref{w22}, but not \eqref{w23}. Hence it does not define an asymp\-to\-ti\-cally unbiased estimator for Boolean models, but it does in the design based setting of Section \ref{euler}. For Boolean models, the asymptotic bias is
\begin{equation*}
\lim_{a\to 0}E\hat{V}_0-\altoverline{V}_0=\big(\tfrac{2-4\sqrt{2}}{\pi}+1\big)\tfrac{\gamma^2}{\pi}EV_1(C)^2e^{-\gamma EV_2(C)}\approx -0,164\tfrac{\gamma^2}{\pi}EV_1(C)^2e^{-\gamma EV_2(C)}.
\end{equation*}

The estimator for the Euler characteristic suggested in \cite{bieri} corresponds to the weights
\begin{equation*}
w^{(0)} = \big(0,\tfrac{1}{4},0,-\tfrac{1}{2},-\tfrac{1}{4},0\big).
\end{equation*}
The bias of this estimator is
\begin{equation*}
\lim_{a\to 0} E\hat{V}_0-\altoverline{V}_0=\big(\tfrac{-4}{\pi}+1\big)\tfrac{\gamma^2}{\pi}EV_1(C)^2e^{-\gamma EV_2(C)}\approx -0,273\tfrac{\gamma^2}{\pi}EV_1(C)^2e^{-\gamma EV_2(C)},
\end{equation*}
which is slightly worse.

The conclusion is that for Boolean models, the best of the estimators for $\altoverline{V}_1$ and $\altoverline{V}_0$ listed here are \eqref{cauchy} and \eqref{OMEuler}, respectively. However, the weights in Proposition \eqref{opt1} and \eqref{opt2}, respectively, give better estimators. 

In the design based setting, all of the classical algorithms listed here except \eqref{graph} are equally good when assessed by means of the results of the present paper.

\acks
The author is supported by the Centre for Stochastic Geometry and Advanced Bioimaging, funded by the Villum Foundation.
The author is most grateful to Markus Kiderlen for handing me the ideas for this paper and for many helpful suggestions and final proofreading. 

\end{document}